\documentclass[a4paper,11pt]{article}
\usepackage[latin1]{inputenc}
\usepackage[english]{babel}
\usepackage{amsmath}
\usepackage{amsfonts}
\usepackage{amssymb}
\usepackage{epsfig}
\usepackage{amsopn}
\usepackage{amsthm}
\usepackage{color}
\usepackage{graphicx}
\usepackage{enumerate}
\usepackage{mathrsfs}
\usepackage{cite}
\newtheorem{theorem}{Theorem}[section]

\newtheorem{lemma}[theorem]{Lemma}
\newtheorem{proposition}[theorem]{Proposition}

\newtheorem*{theorem*}{Theorem}
\newtheorem*{lemma*}{Lemma}
\newtheorem*{remark*}{Remark}
\newtheorem*{definition*}{Definition}
\newtheorem*{proposition*}{Proposition}
\newtheorem*{corollary*}{Corollary}
\numberwithin{equation}{section}
\newcommand{\real}{\mathbb{R}}

\let\ced=\c         

\def\qed{\,\unskip\kern 6pt \penalty 500
\raise -2pt\hbox{\vrule \vbox to8pt{\hrule width 6pt
\vfill\hrule}\vrule}\par}
\definecolor{darkblue}{rgb}{0.05, .05, .65}
\definecolor{darkgreen}{rgb}{0.1, .65, .1}
\definecolor{darkred}{rgb}{0.8,0,0}
\newcommand{\beqn}{\begin{equation}}
\newcommand{\eeqn}{\end{equation}}
\newcommand{\bear}{\begin{eqnarray}}
\newcommand{\eear}{\end{eqnarray}}
\newcommand{\bean}{\begin{eqnarray*}}
\newcommand{\eean}{\end{eqnarray*}}
\begin{document}

\title{\huge \bf A porous medium equation with dominating weighted absorption: three types of self-similar solutions}
\author{
\Large Razvan Gabriel Iagar\,\footnote{Departamento de Matem\'{a}tica
Aplicada, Ciencia e Ingenieria de los Materiales y Tecnologia
Electr\'onica, Universidad Rey Juan Carlos, M\'{o}stoles,
28933, Madrid, Spain, \textit{e-mail:} razvan.iagar@urjc.es}
\\[4pt] \Large Diana-Rodica Munteanu\,\footnote{Faculty of Psychology and Educational Sciences, Ovidius University of Constanta, 900527, Constanta, Romania, \textit{e-mail:} diana.rodica.merlusca@gmail.com}\\
}
\date{}
\maketitle

\begin{abstract}
Self-similar solutions to the porous medium equation with dominating spatially inhomogeneous absorption
$$
\partial_tu=\Delta u^m-|x|^{\sigma}u^p, \quad (x,t)\in\real^N\times(0,\infty), \quad N\geq1,
$$
with exponents $1<p<m$ and $\sigma>0$, are classified. Looking for solutions in the form
$$
u(x,t)=t^{-\alpha}f(|x|t^{\beta}), \quad \alpha=\frac{\sigma+2}{\sigma(m-1)+2(p-1)}, \quad \beta=\frac{m-p}{\sigma(m-1)+2(p-1)},
$$
it is shown that all their profiles satisfy the behavior at infinity given by
$$
\lim\limits_{\xi\to\infty}\xi^{\sigma/(p-1)}f(\xi)=\left(\frac{1}{p-1}\right)^{1/(p-1)},
$$
but the solutions strongly differ with respect to their behavior near the origin: there exist a unique solution with $f(0)>0$, $f'(0)=0$, another unique solution such that $f$ presents a \emph{dead-core}; that is, $f\equiv0$ for $\xi\in[0,\xi_0]$ for some $\xi_0>0$, and, finally, there exists $K^*\in(0,\infty)$ such that, for any $K\in(0,K^*)$, there is at least a solution such that
$$
\lim\limits_{\xi\to0}\xi^{-(\sigma+2)/(m-p)}f(\xi)=K.
$$
The large time behavior of general solutions, making strong use of these three types of self-similar solutions, will be addressed in a companion work.
\end{abstract}

\smallskip


\smallskip


\section{Introduction}

This paper is the first of a two-part work aimed at classifying the large time behavior of solutions to the following quasilinear equation with weighted absorption
\begin{equation}\label{eq1}
\partial_tu=\Delta u^m-|x|^{\sigma}u^p, \quad (x,t)\in\real^N\times(0,\infty),
\end{equation}
in the range of exponents
\begin{equation}\label{range.exp}
1<p<m, \quad \sigma>0, \quad N\geq1.
\end{equation}
In this paper, we perform a complete classification of the self-similar solutions to Eq. \eqref{eq1} with respect to the behavior of their profiles at the origin and near infinity, while the proper asymptotic convergence of general solutions to the self-similar ones as $t\to\infty$, depending on the support and properties of the initial condition, will be addressed in a companion work.

Eq. \eqref{eq1} features a competition between the diffusion and the absorption term. While the former preserves the $L^1$-norm of any solution along the evolution, while spreading its mass, the latter introduces a loss of total mass, which is, in our case, stronger in regions with $|x|$ large, by the influence of the weight $|x|^{\sigma}$. As we shall see in the next discussion, the range $1<p<m$ is a very interesting one, since the influence of the absorption is stronger than the one of the diffusion, leading to a number of properties which strongly depart from the complementary range $p\geq m$.

The qualitative properties of absorption-diffusion equations such as \eqref{eq1} depend strongly on the relative position of $m$ and $p$. Indeed, the spatially homogeneous counterpart of Eq. \eqref{eq1}, that is, 
\begin{equation}\label{eq1.hom}
\partial_tu=\Delta u^m-u^p, \quad m>1, \ p>0,
\end{equation}
has been largely studied in the last forty years and three different ranges with significant differences regarding the large time behavior of solutions have been identified:

$\bullet$ in the range $p>m$, it is well-established that the diffusion term plays the main role in the evolution. More precisely, the large time behavior of solutions to Eq. \eqref{eq1.hom} is considered in \cite{KP86, PT86, KU87, KV88, PW88, GV91, Le97, Kwak98, Kwak98b}. In this range there exists a critical exponent (curiously, identical to the celebrated Fujita exponent for the quasilinear reaction-diffusion equation) $p_F=m+2/N$ such that, for $p>p_F$, the large time behavior is expressed only in terms of solutions to the porous medium equation. We can say that the effect of the absorption is almost canceled after large time. In the complementary range $p\in(m,p_F)$ the asymptotic pattern is represented by some specific self-similar solutions depending on the properties of the initial condition. In particular, very interesting solutions nowadays known as very singular have been obtained and analyzed. A general classification of singular solutions in all ranges is available in \cite{KPV89}. The limiting range $p=p_F$, treated in \cite{GV91}, introduces some logarithmic corrections in the time scales, as an effect of the resonance between the diffusion and the absorption terms.

$\bullet$ in the range $1<p<m$ the absorption plays the dominating role, as it can be seen from the property of localization of supports established (in a very general case including Eq. \eqref{eq1}) in \cite{PT85}. More precisely, if the initial condition $u_0$ is compactly supported, then there is $R_0>0$ (not depending on time) such that the support of $u(t)$ is contained in $B(0,R_0)$ for any $t>0$. This property is an obvious manifestation of the dominance of the absorption, since it is well-known that the diffusion alone tends to occupy the entire $\real^N$ as $t\to\infty$. Interesting self-similar solutions have been obtained in \cite{KPV85, MPV91, CV96}, while a description of the large time behavior with a remarkable formation of a boundary layer (still available, to the best of our knowledge, only in dimension $N=1$) follows from the results of \cite{BNP82, CV99}. The limiting range $p=m$ (see \cite{CVW97}) is very interesting as well, since an easy transformation maps Eq. \eqref{eq1.hom} into a Fisher-KPP type equation considered in \cite{KR04}.

$\bullet$ in the range $0<p<1$, also known as \emph{strong absorption}, studying Eq. \eqref{eq1} seems to be very complicated and a full description of the large time behavior appears to be available only when $m+p=2$ in \cite{GV94}. The instantaneous shrinking of supports of solutions is a very unexpected and counter-intuitive property shared by solutions in this range: instead of expanding, the supports of solutions shrink to the origin (and become compact instantaneously if $u_0>0$ in $\real^N$), a property established in \cite{Abd98}. This shrinking goes exactly in the opposite direction to the diffusion, proving the strength of the absorption effect. Moreover, it is well-known that in this range finite time extinction occurs: solutions vanish completely at some time $T\in(0,\infty)$, as all their mass is lost because of the absorption.

\medskip 

In recent years, the first author and his collaborators started a larger project of understanding the effect of the presence of an unbounded weight on the qualitative and dynamical properties of the solutions to Eq. \eqref{eq1}. The range $p>m$ is thus considered in the two previous works by the authors \cite{IM25, IM26}, covering the large time behavior of solutions to the Cauchy problem for Eq. \eqref{eq1} with almost all types of reasonable initial conditions. The critical case $p=m$ has been recently investigated in \cite{IL26}, and it has been observed that, for $\sigma>N(m-1)/(m+1)$, the large time behavior is self-similar, strongly contrasting with the dynamics for $\sigma=0$, which is of Fisher-KPP type after a transformation.

In the strong absorption range $0<p<1$ a number of results have been obtained in works such as \cite{IL23, ILS24, ILS26} (see also references therein), showing that, for $\sigma>0$ sufficiently large, an interesting alternative between positivity for any $t>0$ and decay as $t\to\infty$ on the one hand, and finite time extinction on the other hand, takes place, depending on the initial condition. The latter works complete the study of finite time extinction performed in previous works such as \cite{Belaud01, BeSh22} (see also references therein). 

We thus address here the remaining range $1<p<m$. The present paper is dedicated to the first step towards the large time behavior of solutions, which is a complete and detailed classification of the self-similar solutions to Eq. \eqref{eq1}. As it is well-understood in the complementary ranges and, more general, in nonlinear diffusion equations, these particular solutions are fundamental both in the large time behavior, as patterns for general solutions, as well as in arguments by comparison or as optimizers in inequalities for the solutions. This is why, due to the complexity of the classification, we devote a full work to it. The large time behavior of general solutions, indicating in each case which self-similar solution is chosen depending on the initial condition, will be considered in a second work dedicated to the range \eqref{range.exp}.

\medskip

\noindent \textbf{Main results.} As explained in the previous paragraphs, the goal of this work is to give a classification of the self-similar solutions in the form
\begin{equation}\label{SSS}
u(x,t)=t^{-\alpha}f(|x|t^{\beta}), \quad (x,t)\in\real^N\times(0,\infty)
\end{equation}
to Eq. \eqref{eq1} in the range of exponents \eqref{range.exp}. Introducing the ansatz \eqref{SSS} in \eqref{eq1}, we deduce after direct calculations that the self-similar exponents are
\begin{equation}\label{SSexp}
\alpha=\frac{\sigma+2}{\sigma(m-1)+2(p-1)}, \quad \beta=\frac{m-p}{\sigma(m-1)+2(p-1)},
\end{equation}
while the profiles $f$ satisfy the differential equation
\begin{equation}\label{SSODE}
(f^m)''(\xi)+\frac{N-1}{\xi}(f^m)'(\xi)+\alpha f(\xi)-\beta\xi f'(\xi)-\xi^{\sigma}f^p(\xi)=0,
\end{equation}
where $\xi=|x|t^{\beta}\geq0$.

In the standard case $\sigma=0$, the different types of self-similar solutions are established in a number of works, mainly in dimension $N=1$, although some extensions to dimension $N\geq2$ are available. Thus, self-similar solutions having a specific behavior of $f(\xi)$ as $\xi\to0$ are introduced as solutions of the Cauchy problem with a specific initial condition in \cite{KPV85}, their classification with respect to the precise coefficient of the behavior in the origin being completed in \cite{MPV91}. It is also obvious that, when $\sigma=0$, there exists a unique constant solution
\begin{equation}\label{const.inf}
c^*:=\left(\frac{1}{p-1}\right)^{1/(p-1)}
\end{equation}
to Eq. \eqref{SSODE}, whose importance has been emphasized in the analysis performed in \cite{CV99} as uniform limit of any general solution in any compact subset of its positivity set. Finally, a self-similar solution with a very interesting profile having a dead-core at some $\xi_0>0$ and tending to $c^*$ as $\xi\to\infty$ has been identified in \cite{CV96} and proved to be very interesting for the large time behavior of solutions stemming from suitable initial conditions, with a left interface, in dimension $N=1$, in \cite{CV99}. To the best of our knowledge, a study of the large time behavior of solutions to Eq. \eqref{eq1} in dimension $N\geq2$ is missing from literature.

We generalize thus the previous results to Eq. \eqref{eq1}, directly posed in dimension $N\geq1$ and with an unbounded coefficient $|x|^{\sigma}$ for $\sigma>0$. The influence of the unbounded coefficient can be seen at first in the non-existence of a flat solution, that is, a constant profile similar to $c^*$ solving \eqref{SSODE}. Moreover, from the technical point of view, the combination of both higher-space dimension and coefficient $|x|^{\sigma}$ makes more difficult to employ similar techniques as in the previously mentioned references. 

The first result is not really surprising, in view of precedents such as \cite[Theorem 1.1(a)]{IM25} and \cite[Theorem 2.1(a)]{IM26} dealing with the range $p>m$, and establishes the existence and uniqueness of a positive self-similar solution.
\begin{theorem}\label{th.SSSpos}
Let $m$, $p$ and $\sigma$ be as in \eqref{range.exp}. Then, the profiles of all bounded self-similar solutions to Eq. \eqref{eq1} present the following behavior at infinity:
\begin{equation}\label{beh.inf}
\lim\limits_{\xi\to\infty}\xi^{\sigma/(p-1)}f(\xi)=c^*,
\end{equation}
where $c^*$ is defined in \eqref{const.inf}. Moreover, there is a unique $A^*>0$ and a unique self-similar profile $f^*$ solution to \eqref{SSODE} such that $f^*(0)=A^*$, $(f^*)'(0)=0$ and $f^*$ behaves as in \eqref{beh.inf} as $\xi\to\infty$. Furthermore, $f^*$ is decreasing on $(0,\infty)$.
\end{theorem}
The uniqueness of the admissible behavior \eqref{beh.inf} as $\xi\to\infty$ is an important feature of the range $1<p<m$, departing strongly from the range $p>m$, where different tails as $\xi\to\infty$, as well as compactly supported self-similar solutions, have been identified in \cite{IM25}. It is a rather unexpected fact that there are no compactly supported self-similar solutions in the form \eqref{SSS}, in spite of the localization property established in \cite{PT85}. On the contrary, as we shall see in the next two statements, in the range \eqref{range.exp} the behavior near the origin becomes rather diverse.

The next theorem extends to the $N$-dimensional space and to the more general equation \eqref{eq1} with $\sigma>0$ the outcome of \cite{CV96}. 
\begin{theorem}\label{th.SSSdc}
Let $m$, $p$ and $\sigma$ be as in \eqref{range.exp}. Then there exists a unique $\xi_0\in(0,\infty)$ and a unique self-similar profile $f_*$ solution to \eqref{SSODE} such that it presents a dead-core at $\xi=\xi_0$, in the sense that
\begin{equation}\label{beh.dc}
f_*(\xi)\begin{cases}
        =0, & \mbox{if } 0\leq\xi\leq\xi_0, \\
        >0, & \mbox{if } \xi_0<\xi<\infty,
      \end{cases} \quad (f_*^m)'(\xi_0)=0,
\end{equation}
and it satisfies \eqref{beh.inf} as $\xi\to\infty$. Moreover, $f_*$ has a single point of maximum on $(\xi_0,\infty)$.
\end{theorem}
Let us stress here that the condition $(f_*^m)'(\xi_0)=0$ at a zero of $f_*$ is a requirement for a solution to a porous medium equation, as explained in \cite[Section 9.8]{VPME}. One can check that, if this condition is not fulfilled, the extension of $f_*$ by zero outside the support becomes only a subsolution (in weak sense). We omit these details here, but they will be employed in the future companion paper on the large time behavior of the general solutions to Eq. \eqref{eq1}.

The unique self-similar solution $U_*$ in the form \eqref{SSS} with the profile $f_*$ given by Theorem \ref{th.SSSdc} has interesting properties. For example, if we let $t\to0$ we observe that, for any $x\in\real^N$, we have
$$
\lim\limits_{t\to0}U_*(x,t)=\lim\limits_{t\to0}t^{-\alpha}f_*(|x|t^{\beta})=0.
$$
Thus, the solution $U_*$ has as initial trace a mass concentrated at infinity: indeed, note that, at least for small $\sigma$, its $L^1$-norm is infinite at any $t>0$, but even when $f_*$ is integrable at infinity (that is, when $\sigma>N(p-1)$, taking into account the behavior \eqref{beh.inf}), a simple calculation gives that 
$$
\int_{\real^N}U_*(x,t)\,dx=t^{-\alpha-N\beta}\int_{\real}f_*(\xi)\,d\xi=t^{-\alpha-N\beta}\|f_*\|_1,
$$
and it is obvious that the latter expression tends to $+\infty$ as $t\to0$. Thus, this very interesting solution $U_*$ is exactly the opposite of the more usual very singular solutions concentrated in the origin: it is a kind of \emph{very singular solution concentrated at infinity} and, at $t>0$, focuses towards the origin (reaching it in infinite time). 

We are left with the third and last type of self-similar solutions. In order to present them, let us observe that Eq. \eqref{eq1} admits an unbounded, explicit stationary solution
\begin{equation}\label{stat.sol}
S(x)=K_0|x|^{(\sigma+2)/(m-p)}, \quad K_0:=\left[\frac{(m-p)^2}{m(\sigma+2)[m(N+\sigma)-p(N-2)]}\right]^{1/(m-p)}.
\end{equation}
The fact that $S$ is a solution to Eq. \eqref{eq1} follows by direct calculation. However, since this stationary solution can be seen as a self-similar one in the form \eqref{SSS} as well, it is natural to ask ourselves whether its behavior as $|x|\to0$ can be taken also by some bounded self-similar profiles. The answer is given in our last (but not least with respect to interest) main result.
\begin{theorem}\label{th.SSSzero}
Let $m$, $p$ and $\sigma$ be as in \eqref{range.exp}. Then there exists $K^*\in(0,\infty]$ such that, for any $K\in(0,K^*)$, there are self-similar profiles such that
\begin{equation}\label{beh.zero}
\lim\limits_{\xi\to0}\xi^{-(\sigma+2)/(m-p)}f(\xi)=K
\end{equation}
and they satisfy \eqref{beh.inf} as $\xi\to\infty$. Furthermore, $K^*>K_0$ and any profile satisfying \eqref{beh.zero} and \eqref{beh.inf} has a single maximum point.
\end{theorem}

\noindent \textbf{Remark.} We expect $K^*<\infty$, as it occurs for $\sigma=0$ and in dimension $N=1$ (see \cite{KPV85}). However, we were unable to prove the non-existence of any bounded self-similar solution satisfying \eqref{beh.zero} and \eqref{beh.inf} for $K$ very large with our techniques.

\medskip 

We thus find that there are actually \emph{many} bounded self-similar profiles having a similar behavior in the origin as the stationary solution \eqref{stat.sol}. It should be moreover noted that the interval of values of $K$ for which the existence of profiles as in Theorem \ref{th.SSSzero} is granted exceeds the explicit constant $K_0$ given in \eqref{stat.sol}. This classification extends the one established in dimension $N=1$ and for $\sigma=0$ in \cite{KPV85, MPV91}.

As we shall see in a forthcoming companion paper, each of these three types of self-similar profiles give rise to self-similar solutions having their own ``basin of attraction" consisting in suitable initial conditions $u_0$ such that the unique solution to Eq. \eqref{eq1} with $u(x,0)=u_0(x)$, $x\in\real^N$, converges to the self-similar solution as $t\to\infty$ in the self-similar time scale. Thus, all these solutions play their role in the dynamical analysis of the equation. Let us also note that, due to the influence of the weight $|x|^{\sigma}$, the constant behavior as $\xi\to\infty$ (typical for $\sigma=0$) is replaced by the specific tail behavior \eqref{beh.inf}, which depends on $\sigma$. 

Let us finally comment that the different profiles in Theorem \ref{th.SSSzero} are a manifestation of the degeneracy of the equation, producing non-uniqueness of solutions to Eq. \eqref{SSODE} when the initial condition is $f(0)=0$, $(f^m)'(0)=0$. With respect to the non-uniqueness of solutions, it has been proved in \cite[Theorem D]{MPV91} that there are infinitely many self-similar profiles satisfying \eqref{beh.zero} together with an unbounded condition (thus, different from \eqref{beh.inf}) as $\xi\to\infty$. In change, there is no proof of either uniqueness or non-uniqueness of the profiles satisfying simultaneously \eqref{beh.zero} and \eqref{beh.inf}, but the uniqueness of such profiles, once fixed $K\in(0,K^*)$, is highly expected.

The existence part of the proof of all the three theorems is based on the study of an alternative formulation of Eq. \eqref{SSODE} in the form of a three-dimensional autonomous dynamical system obtained by a transformation also employed by the authors in \cite{IM25}. The uniqueness of the profile in Theorem \ref{th.SSSpos} follows completely similarly as in \cite{IM26}, while the proof of the uniqueness of the profile in Theorem \ref{th.SSSdc} no longer follows from the study of \eqref{SSODE} or the associated autonomous dynamical system, but employs instead partial differential equations techniques applied to the original equation \eqref{eq1}.

\section{An alternative formulation. Critical points}\label{sec.syst}

In this section, we introduce a transformation mapping the differential equation \eqref{SSODE} into a three-dimensional autonomous dynamical system, which will be the alternative formulation employed throughout the paper. We thus define
\begin{equation}\label{PSchange}
X(\eta):=\frac{m}{\alpha}\xi^{-2}f^{m-1}(\xi),\quad Y(\eta):=\frac{m}{\alpha}\xi^{-1}f^{m-2}(\xi)f'(\xi), \quad Z(\eta):=\frac{1}{\alpha}\xi^{\sigma}f^{p-1}(\xi),
\end{equation}
where we recall that $\alpha$ and $\beta$ are defined in \eqref{SSexp} and $\eta$ is a new independent variable introduced implicitly by
\begin{equation}\label{ind.var}
\frac{d\eta}{d\xi}=\frac{\alpha}{m}\xi f(\xi)^{1-m}.
\end{equation}
By performing direct calculations, we find
\begin{equation}\label{interm0}
\begin{split}
&f'(\xi)=\frac{\alpha}{m}\xi Y(\eta)f(\xi)^{2-m}, \quad (f^m)'(\xi)=\alpha\xi Y(\eta)f(\xi),\\
&(f^m)''(\xi)=\alpha\left(\xi\frac{d}{d\xi}Y(\eta)f(\xi)+\frac{\alpha}{m}\xi^2f(\xi)^{2-m}Y(\eta)^2+Y(\eta)f(\xi)\right),
\end{split}
\end{equation}
and, by replacing the previous expressions in the differential equation \eqref{SSODE}, we are left after algebraic manipulations with the following system (where the first and third equations are derived directly from the definition of $X$ and $Z$ by applying the chain rule)
\begin{equation}\label{PSsyst}
\left\{\begin{array}{ll}\dot{X}=X[(m-1)Y-2X],\\
\dot{Y}=-Y^2+\frac{\beta}{\alpha}Y-X-NXY+XZ,\\
\dot{Z}=Z[(p-1)Y+\sigma X],\end{array}\right.
\end{equation}
where the dot derivatives are taken with respect to $\eta$. The transformation \eqref{PSchange} has been employed by the authors as a secondary one in their previous work \cite{IM25}. Note that the non-negativity of $f$ gives $X(\eta)\geq0$ and $Z(\eta)\geq0$ for any $\eta\in\real$, while the planes $X=0$ and $Z=0$ are invariant for the system \eqref{PSsyst}. We thus work throughout the paper in the region 
$$
(X,Y,Z)\in[0,\infty)\times\mathbb{R}\times[0,\infty).	$$
The critical points of the system \eqref{PSsyst} contained in this region are
\begin{equation}\label{critp}
Q_0=(0,0,0), \quad Q_1=\left(0,\frac{\beta}{\alpha},0\right), \quad Q_{\gamma}=(0,0,\gamma), \quad \gamma>0.
\end{equation}

\subsection{Analysis of the critical points}

We first analyze the flow of the system \eqref{PSsyst} in a neighborhood of the critical points listed in \eqref{critp}.
\begin{lemma}\label{lem.Q0}
The critical point $Q_0$ is a non-hyperbolic point with a one-dimensional unstable manifold and two-dimensional center manifolds with unstable flow, composing a center-unstable three dimensional manifold. The trajectories going out of it correspond to profiles satisfying \eqref{beh.zero} as $\xi\to0$.
\end{lemma}
\begin{proof}
The linearization of the system \eqref{PSsyst} in a neighborhood of $Q_0$ has the matrix
$$
M(Q_0)=\left(
         \begin{array}{ccc}
           0 & 0 & 0 \\
           -1 & \frac{\beta}{\alpha} & 0 \\
           0 & 0 & 0 \\
         \end{array}
       \right).
$$
There is thus a one-dimensional unstable manifold corresponding to the eigenvalue $\lambda_2=\beta/\alpha>0$, tangent to the eigenvector $e_2=(0,1,0)$ and thus contained in the $Y$-axis (which is invariant for the system), according to the uniqueness of the unstable manifold in \cite[Theorem 3.2.1]{GH}. In order to study the center manifolds near $Q_0$, we perform the change of variable
$$
W=\frac{\beta}{\alpha}Y-X, \quad {\rm that \ is,} \quad Y=\frac{\sigma+2}{m-p}(W+X),
$$
and after direct calculations obtain the system
\begin{equation}\label{PSsystQ0}
\left\{\begin{array}{ll}\dot{X}=\frac{1}{\beta}X^2+\frac{(m-1)(\sigma+2)}{m-p}XW,\\[1mm]
\dot{W}=\frac{m-p}{\sigma+2}W+\frac{m-p}{\sigma+2}XZ-\frac{[m(N+\sigma)-p(N-2)]}{m-p}X^2-\frac{N(m-p)+(\sigma+2)(m+1)}{m-p}XW-\frac{\sigma+2}{m-p}W^2,\\[1mm]
\dot{Z}=\frac{1}{\beta}XZ+\frac{(p-1)(\sigma+2)}{m-p}ZW.\end{array}\right.
\end{equation}
Looking for a center manifold whose second order expansion in a neighborhood of the origin is
$$
W=aX^2+bXZ+cZ^2+O(|(X,Z)|^3),
$$
according to \cite[Theorem 3, Section 2.5]{Carr}, we deduce by computing the flow of the system \eqref{PSsystQ0} on the previous surface and equating to zero the quadratic terms that
$$
W=\frac{[m(N+\sigma)-p(N-2)](\sigma+2)}{(m-p)^2}X^2-XZ+O(|(X,Z)|^3),
$$
whence the center manifolds in a neighborhood of $Q_0$ have the equation
\begin{equation}\label{cmQ0}
Y=\frac{\sigma+2}{m-p}(X-XZ)+\frac{[m(N+\sigma)-p(N-2)](\sigma+2)^2}{(m-p)^3}X^2+O(|(X,Z)|^3).
\end{equation}
The direction of the flow of the system \eqref{PSsyst} on these center manifolds is given by the reduced system (according to \cite[Theorem 2, Section 2.4]{Carr})
\begin{equation}\label{redsyst}
\left\{\begin{array}{ll}\dot{X}=\frac{1}{\beta}X^2+O(|(X,Z)|^3),\\[1mm]
\dot{Z}=\frac{1}{\beta}XZ+O(|(X,Z)|^3),\end{array}\right.
\end{equation}
thus any center manifold has an unstable direction of the flow. Together with the unique unstable manifold contained in the $Y$-axis, we obtain a three-dimensional center-unstable manifold in a neighborhood of $Q_0$. In order to deduce the local behavior of the profiles corresponding to the orbits contained in this center-unstable manifold, it is sufficient to note that
$$
\frac{m-p}{\sigma+2}Y-X=W=O(|(X,Z)|^2),
$$
hence, by undoing \eqref{PSchange} and performing easy manipulations,
\begin{equation}\label{interm2}
\lim\limits_{\eta\to-\infty}\frac{Y(\eta)}{X(\eta)}=\frac{\sigma+2}{m-p}.
\end{equation}
We first have to check that the limit as $\eta\to-\infty$ translates into a limit as $\xi\to0$ when undoing the change of variable \eqref{PSchange}. To this end, we observe that, by inverting the implicit definition of $\eta$ in \eqref{ind.var}, we have
\begin{equation}\label{interm1}
\frac{\xi'(\eta)}{\xi(\eta)}=X(\eta).
\end{equation}
In a neighborhood of $Q_0$ and on its center manifold, the local behavior of $X(\eta)$ is given by the first equation of the reduced system \eqref{redsyst}, which readily gives by integration
$$
X(\eta)\sim-\frac{\beta}{\eta}, \quad {\rm as} \ \eta\to-\infty,
$$
and the latter expression, together with \eqref{interm1}, lead to $\xi(\eta)\to0$ as $\eta\to-\infty$. We thus deduce from \eqref{interm2} that
$$
\lim\limits_{\xi\to0}\frac{\xi f'(\xi)}{f(\xi)}=\frac{\sigma+2}{m-p},
$$
hence
$$
\lim\limits_{\xi\to0}\frac{(\ln\,f(\xi))'}{(\ln\,\xi)'}=\frac{\sigma+2}{m-p},
$$
and an application of the L'Hopital's rule readily leads to the local behavior \eqref{beh.zero}, completing the proof.
\end{proof}
We next analyze the flow of the system \eqref{PSsyst} in the neighborhood of the critical point $Q_1$.
\begin{lemma}\label{lem.Q1}
The critical point $Q_1$ is a (hyperbolic) saddle point, with a one-dimensional stable manifold contained in the $Y$-axis and a two-dimensional unstable manifold with a first order approximation given by the one-parameter family
\begin{equation}\label{var.Q1}
(r_C): \left\{\begin{array}{ll}Y(\eta)=\frac{\beta}{\alpha}-\frac{\alpha}{m\beta}\left(1+\frac{N\beta}{\alpha}\right)X(\eta)+o(|X(\eta),Z(\eta)|),\\
Z(\eta)=CX(\eta)^{(p-1)/(m-1)}+o\left(X(\eta)^{(p-1)/(m-1)}\right),\end{array}\right.
\end{equation}
as $\eta\to-\infty$, for arbitrary $C\in[0,\infty)$. The trajectories contained in the unstable manifold of $Q_1$ correspond to profiles presenting the dead-core behavior \eqref{beh.dc}, with the left edge of the support $\xi_0(C)\in(0,\infty)$ given by the following correspondence with $C\in(0,\infty)$:
\begin{equation}\label{corresp}
\xi_0(C)=\left(\frac{m}{\alpha}\right)^{(p-1)/L}(C\alpha)^{(m-1)/L}, \quad L=\sigma(m-1)+2(p-1).
\end{equation}
Moreover, we have the following more precise local behavior at the interface point:
\begin{equation}\label{beh.interf}
f(\xi)\sim\left[\frac{\beta(m-1)}{m}\xi_0\right]^{1/(m-1)}(\xi-\xi_0)^{1/(m-1)}
\end{equation}
as $\xi\to\xi_0$, $\xi>\xi_0$.
\end{lemma}
\begin{proof}
The linearization of the system \eqref{PSsyst} in a neighborhood of $Q_1$ has the matrix
$$
M(Q_1)=\left(
         \begin{array}{ccc}
           \frac{(m-1)\beta}{\alpha} & 0 & 0 \\[1mm]
           -1-\frac{N\beta}{\alpha} & -\frac{\beta}{\alpha} & 0 \\[1mm]
           0 & 0 & \frac{(p-1)\beta}{\alpha} \\
         \end{array}
       \right),
$$
having thus one negative eigenvalue $\lambda_2=-\beta/\alpha$ with corresponding eigenvector $e_2=(0,1,0)$ and two positive eigenvalues with corresponding eigenvectors
\begin{equation}\label{eigen.Q1}
\lambda_1=\frac{(m-1)\beta}{\alpha}, \ e_1=\left(1,-\frac{\alpha+N\beta}{m\beta},0\right), \quad \lambda_3=\frac{(p-1)\beta}{\alpha}, \ e_3=(0,0,1).
\end{equation}
Thus, there exists a two-dimensional unstable manifold in a neighborhood of $Q_1$ tangent to the plane spanned by the eigenvectors $e_1$ and $e_3$, according to the Stable Manifold Theorem \cite[Theorem, Section 2.7]{Pe}. Taking into account the expression of $e_1$ and $e_3$ in \eqref{eigen.Q1}, we readily deduce the first asymptotic equality in \eqref{var.Q1}. In order to derive the connection between $Z(\eta)$ and $X(\eta)$, we translate $Q_1$ to the origin by setting $Y=\overline{Y}+\beta/\alpha$ and note that the first and third equations of the system \eqref{PSsyst} become
$$
\left\{\begin{array}{ll}\dot{X}=\frac{(m-1)\beta}{\alpha}X+X[(m-1)\overline{Y}-2X],\\
\dot{Z}=\frac{(p-1)\beta}{\alpha}Z+Z[(p-1)\overline{Y}+\sigma X]\end{array}\right.
$$
Thus, a simple integration of the linear approximation in the origin of the previous system leads to the second asymptotic equality in \eqref{var.Q1}. With respect to the local behavior of the profiles, we first observe from the latter system that, in a neighborhood of $Q_1$, we have
$$
X(\eta)\sim Ce^{(m-1)\beta\eta/\alpha}, \quad {\rm as} \ \eta\to-\infty, \quad C>0,
$$
and we infer from \eqref{interm1} and a direct integration that $\xi(\eta)\to\xi_0\in(0,\infty)$ as $\eta\to-\infty$ on the trajectories contained in the unstable manifold of $Q_1$. By undoing the change of variable \eqref{PSchange} on the fact that $X(\eta)\to0$ as $\eta\to-\infty$ and on the first equation of \eqref{var.Q1}, we obtain
$$
\lim\limits_{\xi\to\xi_0}\xi^{-2}f^{m-1}(\xi)=0, \quad \lim\limits_{\xi\to\xi_0}\frac{m}{(m-1)\alpha\xi}(f^{m-1})'(\xi)=\frac{\beta}{\alpha},
$$
which gives first that $f(\xi_0)=0$ and then the behavior \eqref{beh.interf}. Indeed, we deduce from L'Hopital's rule that
$$
\lim\limits_{\xi\to\xi_0}\frac{f^{m-1}(\xi)}{\xi-\xi_0}=\lim\limits_{\xi\to\xi_0}(f^{m-1})'(\xi)=\frac{\beta(m-1)}{m}\xi_0,
$$ 
which is equivalent to \eqref{beh.interf}. In particular, we also have
$$
(f^m)'(\xi_0)=\frac{m}{m-1}f(\xi_0)(f^{m-1})'(\xi_0)=0,
$$
which establishes the behavior \eqref{beh.dc} in a neighborhood of $\xi_0$. Finally, fix $C\in(0,\infty)$ and denote by $\xi_0(C)$ the left edge of the support of the profile corresponding to the trajectory $r_C$ in the notation of \eqref{var.Q1}. We deduce from the second equality in \eqref{var.Q1} and \eqref{PSchange} that
\begin{equation}\label{interm3}
\frac{1}{\alpha}\xi^{\sigma}f^{p-1}(\xi)\sim C\left(\frac{m}{\alpha}\right)^{(p-1)/(m-1)}\xi^{-2(p-1)/(m-1)}f^{p-1}(\xi),
\end{equation}
as $\xi\to\xi_0(C)$, $\xi>\xi_0(C)$. The correspondence \eqref{corresp} readily follows from \eqref{interm3} and the positivity of $f(\xi)$ in a right neighborhood of $\xi_0(C)$, completing the proof.
\end{proof}
We are left with the analysis of the critical points $Q_{\gamma}$ with $\gamma\in(0,\infty)$, which is gathered in the next result.
\begin{lemma}\label{lem.Qg}
There exists a unique value
$$
\gamma_0=\frac{1}{\alpha(p-1)}>1
$$
such that the critical point $Q_{\gamma_0}$ has a unique two-dimensional center manifold with stable direction of the flow. The trajectories entering $Q_{\gamma_0}$ correspond to profiles with the behavior \eqref{beh.inf} as $\xi\to\infty$. For any $\gamma\in(0,\infty)\setminus\{\gamma_0\}$, there are no trajectories connecting to $Q_{\gamma}$ outside the invariant plane $X=0$.
\end{lemma}
\begin{proof}
The proof follows similar ideas as the one of \cite[Lemma 2.4]{ILS24} and is very technical. Thus, for the reader's convenience, it is divided into several steps.

\medskip

\noindent \textbf{Step 1. Translation at the origin.} Pick $\gamma\in(0,\infty)$. We perform first the change of variable $Z=T+\gamma$, which maps $Q_{\gamma}$ to the origin of the following system:
\begin{equation}\label{syst.interm}
\left\{\begin{array}{ll}\dot{X}=X[(m-1)Y-2X],\\\dot{Y}=-Y^2-NXY+(\gamma-1)X+\frac{m-p}{\sigma+2}Y+XT,\\
\dot{T}=(T+\gamma)[(p-1)Y+\sigma X].\end{array}\right.
\end{equation}
Let us note that the linearization of the system \eqref{syst.interm} has a matrix with eigenvalues $\lambda_1=\lambda_3=0$ and $\lambda_2=(m-p)/(\sigma+2)>0$, thus we have to work further in order to study the two-dimensional center manifold corresponding to the zero eigenvalues.

\medskip

\noindent \textbf{Step 2. Canonical form.} We next perform the following change of variable, with the aim of putting the system \eqref{syst.interm} in a canonical form for the application of the center manifold theorem. On the one hand, in the second equation of \eqref{syst.interm} we introduce the new variable
$$
W=\frac{m-p}{\sigma+2}Y+(\gamma-1)X, \quad {\rm that \ is}, \quad Y=\frac{\sigma+2}{m-p}[W+(1-\gamma)X].
$$
On the other hand, we also have to remove the linear terms in the last equation of \eqref{syst.interm} and we start by the term in $Y$, thus we replace $T$ by the new variable
$$
V=T-kY, \quad k=\frac{(p-1)(\sigma+2)\gamma}{m-p},
$$
where the choice of $k$ is obtained by calculations in order to achieve the goal of reducing the linear term in $Y$ (for a more detailed calculation at this point, see for example the analogous one in the proof of \cite[Lemma 2.4]{IS21}). Performing all these changes of variable at the same time, we obtain the following (rather complicated) system:
\begin{equation}\label{syst.interm2}
\left\{\begin{array}{ll}\dot{X}=\left[\frac{(m-1)(\sigma+2)(1-\gamma)}{m-p}-2\right]X^2+\frac{(m-1)(\sigma+2)}{m-p}XW,\\
\dot{W}=\frac{m-p}{\sigma+2}W-\frac{\sigma+2}{m-p}W^2+\frac{m-p}{\sigma+2}VX+\frac{D_1}{m-p}X^2-\frac{D_2}{m-p}XW,\\
\dot{V}=[\gamma\sigma+k(1-\gamma)]X-\frac{D_3}{(m-p)^3}X^2+\frac{D_4}{m-p}XV+O(|(X,V)|)W,\end{array}\right.
\end{equation}
where we omit for simplicity the terms in $W$ in the last equation, as they will not be important in the sequel, and with
\begin{equation*}
\begin{split}
&D_1=(\gamma-1)[(N-2)(m-p)+m(1-\gamma)(\sigma+2)-(\sigma+2)(p-1)\gamma],\\
&D_2=N(m-p)+(\sigma+2)[m+1-\gamma(m+p)],\\
&D_3=(p-1)\gamma(\gamma-1)(\sigma+2)^2[(N+\sigma)(m-p)+(\sigma+2)(p+\gamma-p\gamma)],\\
&D_4=\sigma(m-p)+(p-1)(\sigma+2)(1-2\gamma).
\end{split}
\end{equation*}

\medskip

\noindent \textbf{Step 3. Analysis of the center manifolds.} Any center manifold of the critical point $Q_{\gamma}$ is thus of the form
$$
W=aX^2+bXV+cV^2+o(|(X,V)|^2),
$$
in view of \cite[Theorem 3, Section 2.5]{Carr}, where the coefficients $a$, $b$ and $c$ can be straightforwardly computed by requiring that the terms in the calculation of the flow of the system \eqref{syst.interm2} across the approximation of the center manifold are only of higher order than quadratic. We thus find a rather complicated expression for $a$ (that we omit here), $b=-1$ and $c=0$. It is rather obvious intuitively (due to the lack of such terms in the system \eqref{syst.interm2}) and can be proved by induction that pure powers of $V$ are missing from the expansion of $W$ to any order, thus we can write
$$
W=aX^2-XV+Xo(|(X,V)|).
$$
In order to establish the flow on the center manifolds, we proceed as in \cite[Theorem 2, Section 2.4]{Carr} and derive the reduced system in variables $(X,V)$ on the manifold by substituting $W$ by its expansion in the first and third equations of the system \eqref{syst.interm2} and keeping only the quadratic terms. We thus get
\begin{equation}\label{redsyst2}
\left\{\begin{array}{ll}\dot{X}=\left[\frac{(m-1)(\sigma+2)(1-\gamma)}{m-p}-2\right]X^2+X^2O(|(X,V)|),\\
\dot{V}=[\gamma\sigma+k(1-\gamma)]X-\frac{D_3}{(m-p)^3}X^2+\frac{D_4}{m-p}XV+XO(|(X,V)|^2).\end{array}\right.
\end{equation}

\medskip

\noindent \textbf{Step 4. Uniqueness of $\gamma_0$.} We are left with analyzing the quadratic system \eqref{redsyst2}. We observe that, if we consider only trajectories arriving from outside the plane $X=0$, we can ``simplify" by $X$, in the sense that, if we define a new independent variable $\nu$ such that
\begin{equation}\label{interm4}
\frac{d}{d\nu}=\frac{1}{X}\frac{d}{d\eta},
\end{equation}
the system \eqref{redsyst2} becomes
\begin{equation}\label{redsyst3}
\left\{\begin{array}{ll}\frac{dX}{d\nu}=\left[\frac{(m-1)(\sigma+2)(1-\gamma)}{m-p}-2\right]X+XO(|(X,V)|),\\
\frac{dV}{d\nu}=\gamma\sigma+k(1-\gamma)-\frac{D_3}{(m-p)^3}X+\frac{D_4}{m-p}V+O(|(X,V)|^2).\end{array}\right.
\end{equation}
A simple integration of the system \eqref{redsyst3} (neglecting the higher order terms) shows that a trajectory of it may pass by the origin if and only if the zero order term in the second equation vanishes, that is,
$$
\gamma\sigma+k(1-\gamma)=0,
$$
which leads to the value of $\gamma_0$ given in the statement. The previous argument proves that, for any $\gamma\in(0,\infty)\setminus\{\gamma_0\}$, all the trajectories connecting to $Q_{\gamma}$ are fully contained in the invariant plane $X=0$ (where the change of independent variable \eqref{interm4} is not applicable). Letting now $\gamma=\gamma_0$, the system \eqref{redsyst3} reduces to
\begin{equation}\label{redsyst30}
\left\{\begin{array}{ll}\frac{dX}{d\nu}=-\frac{\sigma(m-1)+2(p-1)}{p-1}X+XO(|(X,V)|),\\
\frac{dV}{d\nu}=-\frac{D_3}{(m-p)^3}X-\frac{1}{\beta}V+O(|(X,V)|^2).\end{array}\right.
\end{equation}
We easily observe that $(0,0)$ is a stable node for the system \eqref{redsyst30}, whence the flow on the center manifold of $Q_{\gamma_0}$ points in the stable direction. Moreover, since the only nonzero eigenvalue, $\lambda_2=(m-p)/(\sigma+2)$, is positive, standard results of center manifold theory (see for example \cite[Theorem 3.2']{Sij}) establish that the center manifold of $Q_{\gamma_0}$ is unique.

\medskip

\noindent \textbf{Step 5. Behavior of the profiles.} The behavior of the trajectories entering $Q_{\gamma_0}$ on its center manifold is given by $Z(\eta)\to\gamma_0$ as $\eta\to\infty$. Since, in a first approximation, we have in a neighborhood of $Q_{\gamma_0}$ that
$$
\dot{X}(\eta)\sim-\frac{\sigma(m-1)+2(p-1)}{p-1}X^2(\eta), \quad {\rm as} \ \eta\to\infty,
$$
we find that
$$
X(\eta)\sim\frac{p-1}{\sigma(m-1)+2(p-1)}\frac{1}{\eta}, \quad {\rm as} \ \eta\to\infty.
$$
It then follows from \eqref{interm1} and the definition of $X(\eta)$ that
$$
\frac{\xi'(\eta)}{\xi(\eta)}=X(\eta)\sim\frac{p-1}{\sigma(m-1)+2(p-1)}\frac{1}{\eta}, \quad {\rm as} \ \eta\to\infty,
$$
which readily gives that
$$
\xi\sim\eta^{(p-1)/[\sigma(m-1)+2(p-1)]}\to\infty, \quad {\rm as} \ \eta\to\infty.
$$
It thus follows that $\xi\to\infty$ on the profiles corresponding to the trajectories belonging to the center manifold of $Q_{\gamma_0}$, and the precise behavior \eqref{beh.inf} follows from the fact that $Z(\eta)\to\gamma_0$ as $\eta\to\infty$ and by undoing \eqref{PSchange}. Finally, note that $\gamma_0>1$, since
$$
\gamma_0-1=\frac{\sigma(m-p)}{(\sigma+2)(p-1)}>0,
$$
completing the proof.
\end{proof}

\subsection{Some critical points at infinity}

In order to complete the local analysis of the critical points of the system \eqref{PSsyst}, we have to analyze its critical points at infinity, obtained by compactifying the space to the Poincar\'e hypersphere, as indicated in \cite[Section 3.10]{Pe}. We set
$$
X=\frac{\overline{X}}{W}, \qquad Y=\frac{\overline{Y}}{W}, \qquad Z=\frac{\overline{Z}}{W}.
$$
The critical points at infinity of the system \eqref{PSsyst}, expressed in the new variables introduced above, are then given by the following system (according to \cite[Theorem 4, Section 3.10]{Pe}):
\begin{equation*}
\left\{\begin{array}{ll}\overline{X}[\overline{X}\overline{Z}-(N-2)\overline{X}\overline{Y}-m\overline{Y}^2]=0,\\
\overline{X}\overline{Z}[(\sigma+2)\overline{X}+(p-m)\overline{Y}]=0,\\
\overline{Z}[p\overline{Y}^2+(\sigma+N)\overline{X}\overline{Y}-\overline{X}\overline{Z}]=0,\end{array}\right.
\end{equation*}
together with the condition of belonging to the equator of the hypersphere, which implies $W=0$ and thus the additional equation $\overline{X}^2+\overline{Y}^2+\overline{Z}^2=1$. One can thus find several critical points (see for example \cite{IM25}), but we will keep the analysis short by only considering the points presenting interest for our global analysis, which are
$$
P_0=(1,0,0,0), \quad Q_{2,3}=(0,\pm1,0,0),
$$
all the other points being neglected. In order to analyze the critical point $P_0$, we perform a projection on the $X$ variable according to \cite[Theorem 5(a), Section 3.10]{Pe}, and the point $P_0$ is then topologically equivalent to the origin of the system 
\begin{equation}\label{PSsyst2}
\left\{\begin{array}{ll}\frac{dx}{d\theta}=x(2-(m-1)y), \\ \frac{dy}{d\theta}=-x-(N-2)y+z-my^2+\frac{m-p}{\sigma+2}xy, \\ \frac{dz}{d\theta}=z(\sigma+2+(p-m)y).\end{array}\right.
\end{equation}
The system \eqref{PSsyst2} is derived from \eqref{PSsyst} by the change of variable indicated in \cite[Theorem 5(a), Section 3.10]{Pe}, that is,
$$
x=\frac{1}{X}, \quad y=\frac{Y}{X}, \quad z=\frac{Z}{X},
$$
where $\theta=\ln\,\xi$ is the new independent variable of the system \eqref{PSsyst2}. Thus, the local analysis of the flow of the system \eqref{PSsyst} near $P_0$ can be performed on the system \eqref{PSsyst2}. We have the following result:
\begin{lemma}\label{lem.P0}
If $N\geq3$, the critical point $P_0\equiv(0,0,0)$ in the system \eqref{PSsyst2} is a saddle point with a one-dimensional stable manifold contained in the invariant $y$-axis and a two-dimensional unstable manifold tangent to the plane spanned by the eigenvectors
$$
v_z=(N,-1,0), \quad v_x=(0,1,N+\sigma).
$$
For $N=2$, the critical point $P_0$ is a saddle-node, with an unstable manifold tangent to the plane spanned by the eigenvectors $v_z$ and $v_x$ and center manifolds tangent to the $y$-axis. For $N=1$, the critical point $P_0$ is an unstable node. In all these cases, the trajectories contained in the two-dimensional unstable manifold tangent to the plane spanned by $v_z$ and $v_x$ form a one-parameter family given by
\begin{equation}\label{lC}
(l_C): \ y(\theta)\sim-\frac{x(\theta)}{N}, \quad z(\theta)\sim Cx(\theta)^{(\sigma+2)/2}, \quad C\in[0,\infty)
\end{equation}
and correspond to profiles such that $f(0)=A>0$, $f'(0)=0$, where $A$ and the parameter $C$ of the trajectory $l_C$ are related by 
\begin{equation*}
A=(Cm)^{2/L}\left(\frac{\alpha}{m}\right)^{(\sigma+2)/L}, \quad L=\sigma(m-1)+2(p-1).
\end{equation*}
\end{lemma}
Noticing that the only difference between our system \eqref{PSsyst2} and the system investigated in \cite[Section 2.2]{IM25} is a change of sign at a quadratic term in the second equation in the system \eqref{PSsyst2}, which is irrelevant in a local linear approximation, the proof of Lemma \ref{lem.P0} is completely identical to the proof of \cite[Lemma 2.4]{IM25} in dimensions $N\geq2$ and very similar to the one of \cite[Lemma 6.1, Part 1]{IMS23} in dimension $N=1$ and is thus omitted. Let us only observe that the negativity of $y(\theta)$ in a neighborhood of $P_0$ given by the first equivalence in \eqref{lC} implies that any trajectory $l_C$ with $C\in[0,\infty)$ goes out of $P_0$ into the half-space $y=Y/X<0$.

The analysis of the critical points $Q_2$ and $Q_3$ is gathered in the following statement.
\begin{lemma}\label{lem.Q23}
The critical point $Q_2$ is an unstable node and the critical point $Q_3$ is a stable node. The trajectories stemming from the unstable node $Q_2$ correspond to profiles such that there is $\xi_0\in(0,\infty)$ and $\delta>0$ with
\begin{equation}\label{beh.Q2}
f(\xi_0)=0, \qquad f(\xi)>0 \ {\rm for} \ \xi\in(\xi_0,\xi_0+\delta), \qquad (f^m)'(\xi_0)>0.
\end{equation}
The trajectories entering the stable node $Q_3$ correspond to profiles having a compact support such that there is $\xi_0\in(0,\infty)$ and $\delta\in(0,\xi_0)$ with
\begin{equation}\label{beh.Q3}
f(\xi_0)=0, \qquad f(\xi)>0 \ {\rm for} \ \xi\in(\xi_0-\delta,\xi_0), \qquad (f^m)'(\xi_0)<0.
\end{equation}
\end{lemma}
Note that the behavior near $\xi_0$ of both types of profiles in \eqref{beh.Q2} and \eqref{beh.Q3} is not an interface one, since the condition $(f^m)'(\xi_0)=0$ stemming from \cite[Section 9.8]{VPME} is not fulfilled. This is why, we are not focusing on the trajectories connecting to $Q_2$ or $Q_3$ as main goals of this paper. The proof of Lemma \ref{lem.Q23} follows by employing the change of variable
$$
\overline{x}=\frac{X}{Y}, \quad \overline{y}=\frac{1}{Y}, \quad \overline{z}=\frac{Z}{Y},
$$
as indicated in \cite[Theorem 5(b), Section 3.10]{Pe}, which transforms the system \eqref{PSsyst} into the following system
$$
\left\{\begin{array}{ll}\pm\overline{x}'=\overline{x}\left[m+(N-2)\overline{x}+\overline{x}\overline{y}-\frac{\beta}{\alpha}\overline{y}-\overline{x}\overline{z}\right],\\
\pm\overline{y}'=\overline{y}\left[1+N\overline{x}+\overline{x}\overline{y}-\frac{\beta}{\alpha}\overline{y}-\overline{x}\overline{z}\right],\\
\pm\overline{z}'=\overline{z}\left[p+(\sigma+N)\overline{x}+\overline{x}\overline{y}-\frac{\beta}{\alpha}\overline{y}-\overline{x}\overline{z}\right].
\end{array}\right.
$$
In the previous system, the plus sign corresponds to $Q_2$ and the minus sign corresponds to $Q_3$, seen as the origin in variables $(\overline{x},\overline{y},\overline{z})$. The local analysis is then straightforward and very similar to the one performed in \cite[Lemma 3.3]{IS25}, thus we omit it here.

\section{Preparatory results of global analysis: invariant region and planes}

In this section we gather a number of preparatory results of global analysis in the phase space associated to the system \eqref{PSsyst}. An important step in the global analysis of a phase space is, on the one hand, to derive suitable positively or negatively invariant regions and, on the other hand, to study first the trajectories contained in the limiting invariant planes $X=0$ and $Z=0$, that are borderlines of the stable or unstable manifolds of interest. We first introduce an important positively invariant region.
\begin{proposition}\label{prop.inv}
The region
\begin{equation}\label{invreg}
\mathcal{R}:=\{(X,Y,Z)\in\real^3: X\geq0, Y>0, Z>1\}
\end{equation}
is positively invariant for the system \eqref{PSsyst}.
\end{proposition}
\begin{proof}
The flow of the system \eqref{PSsyst} across the plane $Y=0$ (with normal direction $(0,1,0)$) is given by the sign of the expression $X(1-Z)$, which is non-negative for $Z>1$. The flow of the system \eqref{PSsyst} across the plane $Z=1$ (with normal direction $(0,0,1)$) is given by the sign of the expression $(p-1)Y+\sigma X$, which is positive for $Y>0$. It follows that, once a trajectory entered $\mathcal{R}$, it cannot go out through any of its boundary planes, completing the proof.
\end{proof}
The next step is understanding the invariant plane $X=0$.
\begin{proposition}\label{prop.X0}
The trajectories contained in the plane $X=0$ stemming from the critical points $Q_0$ and $Q_1$ enter the invariant region $\mathcal{R}$.
\end{proposition}
\begin{proof}
We note that, in the reduced system obtained from \eqref{PSsyst} on the plane $X=0$, that is,
\begin{equation*}
\left\{\begin{array}{ll}\dot{Y}=-Y^2+\frac{\beta}{\alpha}Y,\\\dot{Z}=(p-1)YZ,\end{array}\right.
\end{equation*}
the axis $Y=0$ is an invariant, critical line (composed by the critical points $Q_{\gamma}$), thus, it splits the plane into two parts. Moreover, the vertical line $Y=\beta/\alpha$ is obviously the unique trajectory going out of $Q_1$ and enters the invariant region $\mathcal{R}$. All the trajectories going out of $Q_0$ on its center manifolds given by \eqref{cmQ0} enter the half-plane $Y>0$ and thus remain forever in the strip $0<Y<\beta/\alpha$. We observe that in this strip $\dot{Y}>0$ and $\dot{Z}>0$, which prevents the trajectories to get back to a critical point $Q_{\gamma}$ situated on the line $Y=0$. This monotonicity, the non-existence of critical points with $Z\geq1$ in the strip (with the exception of $Q_{\gamma}$) and an application of the Poincar\'e-Bendixon's Theorem (see for example \cite[Section 3.7]{Pe}) imply that all these trajectories have to cross the line $Z=1$ and reach the region $\mathcal{R}$ defined in \eqref{invreg}, as claimed.
\end{proof}
The analysis of the invariant plane $Z=0$ is a bit more involved, but leads to a simple conclusion.
\begin{proposition}\label{prop.Z0}
The unique trajectory stemming from the critical point $Q_1$ and all trajectories stemming from the critical point $Q_0$ in the plane $Z=0$ connect to the stable node $Q_3$ at infinity. The same holds true for the unique trajectory going out of the critical point $P_0$.
\end{proposition}
\begin{proof}
Recall that the system \eqref{PSsyst} reduces in the invariant plane $Z=0$ to
\begin{equation}\label{redZ0}
\left\{\begin{array}{ll}\dot{X}=X[(m-1)Y-2X],\\\dot{Y}=-Y^2-NXY-X+\frac{\beta}{\alpha}Y.\end{array}\right.
\end{equation}
Consider the triangular region
$$
\mathcal{T}:=\left\{(X,Y)\in\real^2: X\geq0, \frac{2X}{m-1}\leq Y\leq\frac{\beta}{\alpha}\right\}.
$$
The flow on the system \eqref{redZ0} across the line $Y=\beta/\alpha$ (with normal direction $(0,1)$) is given by the sign of the expression
$$
-\left(\frac{N\beta}{\alpha}+1\right)X<0,
$$
while the flow of the system \eqref{redZ0} on the line $Y=2X/(m-1)$ (with normal direction $(-2,m-1)$) is given by the sign of the expression
$$
-\frac{(m-1)(mN-N+2)}{2}Y^2-\frac{m-1}{2\alpha}Y<0.
$$
Moreover, the isocline $\dot{Y}=0$, that is,
\begin{equation}\label{iso}
X=\frac{Y}{1+NY}\left(\frac{\beta}{\alpha}-Y\right),
\end{equation}
has a positive hump connecting $Q_0$ and $Q_1$ and lying in the region $\mathcal{T}$. Indeed, the slope of this isocline as $Y\to0$ is given by
$X/Y=\beta/\alpha$ and
$$
\frac{\beta}{\alpha}-\frac{m-1}{2}=\frac{m-p}{\sigma+2}-\frac{m-1}{2}=-\frac{\sigma(m-1)+2(p-1)}{2(\sigma+2)}<0,
$$
which implies that \eqref{iso} enters $Q_0$ through the interior of the triangular region $\mathcal{T}$. Furthermore, the flow of the system \eqref{PSsyst} across the curve \eqref{iso} is given by the sign of $(m-1)Y-2X$, which is positive in the hump connecting $Q_1$ to $Q_0$ and negative in the second branch of this curve, lying in the half-plane $Y<-1/N$.

It is then clear from \eqref{cmQ0} and the first expansion of \eqref{var.Q1} that the trajectories going out of $Q_0$, respectively $Q_1$, enter the region $\mathcal{T}$ and an application of the Poincar\'e-Bendixon's Theorem in the compact region $\mathcal{T}$ together with the non-existence of critical points with stable manifolds in the closure of $\mathcal{T}$ prove that all these trajectories have to leave the region $\mathcal{T}$. This is only possible by crossing the edge $(m-1)Y-2X=0$ from inside to outside $\mathcal{T}$, that is, entering the region where $(m-1)Y<2X$. Since the positive hump of \eqref{iso} is contained in the closure of $\mathcal{T}$, we infer that, after crossing this line, the trajectories enter a region with $\dot{X}<0$ and $\dot{Y}<0$. The flow of the system across the second branch of the isocline \eqref{iso} proves that it cannot be crossed from outside and thus all the trajectories leaving $Q_1$ and $Q_0$ will satisfy $\dot{X}<0$ and $\dot{Y}<0$ forever after leaving the region $\mathcal{T}$. Thus, there are
$$
X_{\infty}:=\lim\limits_{\eta\to\eta^+}X(\eta), \quad Y_{\infty}:=\lim\limits_{\eta\to\eta^+}Y(\eta)
$$
along any of these trajectories, where $\eta^+$ is the upper edge of the maximal interval of definition of the trajectory under consideration. Since $X\geq0$ $\dot{X}\leq0$ along these trajectories, we deduce that $0\leq X_{\infty}<\infty$ and thus $Y_{\infty}=-\infty$ (as otherwise $(X_{\infty},Y_{\infty})$ would be a finite critical point attracting the trajectory, and there is no such point). The latter limits show that any of the trajectories analyzed here enter the stable node $Q_3$.

We are left with analyzing the unique trajectory leaving $P_0$ contained in the plane $Z=0$. We infer from \eqref{lem.P0} that this trajectory leaves $P_0$ tangent to the eigenvector $v_1=(N,-1,0)$ in the variables of the system \eqref{PSsyst2}, that is, with $y/x=Y=-1/N$. It is then easy to notice by comparing the isocline of the system \eqref{PSsyst2} for $z=0$ with the eigenvector $v_1$ that it also starts with $\dot{y}<0$, which implies in a neighborhood of $P_0$ that
$$
\dot{Y}\sim\frac{d}{d\theta}\left(\frac{y}{x}\right)=\frac{y'-2y+(m-1)y^2}{x}<0,
$$
thus the trajectory leaves $P_0$ with $\dot{Y}<0$. Moreover, $\dot{X}<0$ is ensured for any $Y<0$. The previous considerations thus also apply to this trajectory, proving that it reaches the stable node $Q_3$, as stated.
\end{proof}
These preparatory results allow us to start proving our main theorems.

\section{Proof of Theorem \ref{th.SSSdc}: existence}\label{sec.existQ1}

We first prove the existence statement in Theorem \ref{th.SSSdc}. The reason for starting with it instead of following the order listed in the Introduction is that its proof does not involve critical points at infinity and another system than \eqref{PSsyst}, while the forthcoming proof of Theorem \ref{th.SSSpos} becomes almost a subset of the proof that follows below.
\begin{proof}[Proof of Theorem \ref{th.SSSdc}: existence]
We have to show that there exists a trajectory of the system \eqref{PSsyst} connecting $Q_1$ to $Q_{\gamma_0}$. The strategy of the proof is to perform a shooting on the unstable manifold of the critical point $Q_1$, given in \eqref{var.Q1}, the shooting parameter being $C\in[0,\infty)$. Let us also denote by $r_{\infty}$ the unique trajectory of this manifold contained in the invariant plane $X=0$, which is the line $Y=\beta/\alpha$, as shown in Proposition \ref{prop.X0}. Introduce then the following sets:
\begin{equation}\label{sets}
\begin{split}
&\mathcal{A}:=\{C\in(0,\infty): {\rm the \ trajectory} \ r_C \ {\rm connects \ to} \ Q_3\},\\
&\mathcal{C}:=\{C\in(0,\infty): {\rm the \ trajectory} \ r_C \ {\rm enters \ the \ region} \ \mathcal{R}\},\\
&\mathcal{B}:=(0,\infty)\setminus(\mathcal{A}\cup\mathcal{C}),
\end{split}
\end{equation}
where we recall that $\mathcal{R}$ is the positively invariant region defined in \eqref{invreg}. On the one hand, we infer from Proposition \ref{prop.Z0} that the trajectory $r_0$ connects to $Q_3$ and, due to the stability of $Q_3$, the set $\mathcal{A}$ is open, non-empty and more precisely there is $C_1\in(0,\infty)$ such that $(0,C_1)\subseteq\mathcal{A}$. On the other hand, we infer from Proposition \ref{prop.X0} that the trajectory $r_{\infty}$ enters the invariant region $\mathcal{R}$. Moreover, the condition of entering $\mathcal{R}$ implies that $\mathcal{C}$ is an open set and there is $C_2>0$ such that $(C_2,\infty)\subseteq\mathcal{C}$. A standard topological fact then gives that $\mathcal{B}\neq\emptyset$. Pick $C\in\mathcal{B}$. The rest of the proof, analyzing carefully the trajectory $r_C$, is divided into several steps for simplicity.

\medskip

\noindent \textbf{Step 1. The trajectory $r_C$ with $C\in\mathcal{B}$ crosses the plane $Y=0$.} Assume for contradiction that this is not the case, that is, the trajectory $l_C$ remains forever in the region
$$
\mathcal{R}_0:=\{(X,Y,Z)\in\real^3: X\geq0, Y\geq0, 0\leq Z\leq 1\}.
$$
The third equation of the system \eqref{PSsyst} implies then that $\eta\mapsto Z(\eta)$ is increasing along this trajectory, and thus there is $Z_0=\lim\limits_{\eta\to\infty}Z(\eta)\in(0,1]$. Let then $(\widetilde{X},\widetilde{Y},\widetilde{Z})$ be a point of the $\omega$-limit set of the trajectory $r_C$, that is, there exists a sequence $\{\eta_{k}\}_{k\geq1}$ such that
$$
\lim\limits_{k\to\infty}\eta_k=\infty, \quad \lim\limits_{k\to\infty}(X(\eta_k),Y(\eta_k),Z(\eta_k))=(\widetilde{X},\widetilde{Y},\widetilde{Z}).
$$
Then, necessarily, $\widetilde{Z}=Z_0$ and thus the $\omega$-limit set of the trajectory $r_C$ is included in the plane $Z=Z_0$ for some $Z_0\in(0,1]$. We also know from \cite[Theorem 2 and Corollary, Section 3.2]{Pe} that the $\omega$-limit set of the trajectory $r_C$ is either a critical point or a trajectory of the system \eqref{PSsyst}, in both cases with $Z=Z_0$ constant. The third equation of the system \eqref{PSsyst} then gives that $\widetilde{X}=\widetilde{Y}=0$, which implies that the trajectory $r_C$ enters the critical point $Q_{Z_0}$. But this is a contradiction with Lemma \ref{lem.Qg}, since the only critical point allowing for a stable manifold is $Q_{\gamma_0}$ with $\gamma_0>1$. This contradiction proves that the trajectory $r_C$ crosses the plane $Y=0$ at a point $(X(\eta_0),0,Z(\eta_0))$ with $Z(\eta_0)<1$.

\medskip

\noindent \textbf{Step 2. The trajectory $r_C$ is bounded in $X$.} It is easy to show that $Y(\eta)<0$ for any $\eta>\eta_0$ on the trajectory $r_C$. Indeed, by the flow of the system \eqref{PSsyst} across the plane $Y=0$ (see Proposition \ref{prop.inv}), if the trajectory $r_C$ crosses again the plane $Y=0$ in order to return to the half-space $Y>0$, it directly enters the positively invariant region $\mathcal{R}$, which is a contradiction since $C\not\in\mathcal{C}$. Thus, denoting by $\eta^+$ the upper edge of the maximal interval of definition of the trajectory $r_C$, it follows that $Y(\eta)<0$ for any $\eta\in(\eta_0,\eta^+)$ and, from the first equation of the system \eqref{PSsyst}, that also $\dot{X}(\eta)<0$ for any $\eta\in(\eta_0,\eta^+)$. In particular, this means that $0<X(\eta)<X(\eta_0)$ for $\eta\in(\eta_0,\eta^+)$.

\medskip

\noindent \textbf{Step 3. If $\eta\mapsto Z(\eta)$ is increasing and bounded we reach the conclusion.} Assume first that the trajectory $r_C$ remains in the region
$$
\mathcal{R}_2:=\{(X,Y,Z)\in\real^3: X\geq0, Y<0, (p-1)Y+\sigma X>0\},
$$
for any $\eta\in(\eta_0,\eta^+)$. The third equation in the system \eqref{PSsyst} ensures then that $\eta\mapsto Z(\eta)$ is increasing on $(\eta_0,\eta^+)$, while the definition of $\mathcal{R}_2$ and the fact that $X(\eta)<X(\eta_0)$ for $\eta>\eta_0$ also gives that
\begin{equation}\label{interm5}
-\frac{\sigma}{p-1}X(\eta_0)<Y(\eta)<0, \quad \eta>\eta_0.
\end{equation}
Moreover, there exist
$$
X_{\infty}:=\lim\limits_{\eta\to\eta^+}X(\eta)\in[0,X(\eta_0)), \quad Z_{\infty}:=\lim\limits_{\eta\to\eta^+}Z(\eta)>Z(\eta_0).
$$
If $0<Z_{\infty}<\infty$, we continue with similar arguments as in Step 1 and prove that the $\omega$-limit set of the trajectory $r_C$ is a critical point of the form $(X_{\infty},Y_{\infty},Z_{\infty})$ with
$$
Y_{\infty}=-\frac{\sigma}{p-1}X_{\infty},
$$
and this is only possible if $X_{\infty}=Y_{\infty}=0$ and $Z_{\infty}=\gamma_0$, according to Lemma \ref{lem.Qg}, leading to the conclusion.

\medskip

\noindent \textbf{Step 4. If $\eta\mapsto Z(\eta)$ is increasing, then it is bounded.} Assume next for contradiction that we are under the assumptions in Step 3, but with $Z_{\infty}=\infty$. We have then two possibilities:

$\bullet$ assume for contradiction that $X(\eta)Z(\eta)$ is unbounded for $\eta>\eta_0$. This leads to $\dot{Y}>0$ on some interval $(\eta_1,\eta^+)$ with $\eta_1\geq\eta_0$ and thus to the boundedness of $Y(\eta)$. By dividing the equation \eqref{SSODE} by $f(\xi)$, we get
$$
\frac{(f^m)''(\xi)}{f(\xi)}+\frac{(N-1)(f^m)'(\xi)}{\xi f(\xi)}+\alpha-\beta\frac{\xi f'(\xi)}{f(\xi)}-\xi^{\sigma}f^{p-1}(\xi)=0
$$
and we further infer from \eqref{interm0} (by using an ``abuse of language" and identifying $Z(\xi)$ with $Z(\eta)$ and similar ones) that
\begin{equation}\label{interm7}
\frac{(f^m)''(\xi)}{f(\xi)}+\frac{(N-1)(f^m)'(\xi)}{\xi f(\xi)}+\alpha-\beta\frac{Y(\xi)}{X(\xi)}-\alpha Z(\xi)=0.
\end{equation}
We first observe that, if we denote by $\xi^+$ the upper edge of the maximal interval of definition of $f$, then necessarily $\xi^+=\infty$. Indeed, since $Z(\eta)\to\infty$ and $X(\eta)\to X_{\infty}\in[0,X(\eta_0))$ as $\eta\to\eta^+$, we find
$$
\lim\limits_{\xi\to\xi^+}\xi^{\sigma}f^{p-1}(\xi)=\infty, \quad \lim\limits_{\xi\to\xi^+}\xi^{-2}f^{m-1}(\xi)=\frac{\alpha}{m}X_{\infty},
$$
and if $\xi^+<\infty$ we obtain simultaneously that $f^{p-1}(\xi)\to\infty$ and $f^{m-1}(\xi)\to \alpha X_{\infty}/m<\infty$ as $\xi\to\xi^+$, which is a contradiction since both $m-1$ and $p-1$ are positive exponents. Thus $\xi^+=\infty$. Note next that $\alpha Z(\xi)$ dominates both terms $\alpha$ and $\beta Y(\xi)/X(\xi)$ in \eqref{interm7} as $\xi\to\infty$, due to the unboundedness of $(XZ)(\xi)$ and the boundedness of $Y(\xi)$. Moreover, $f'(\xi)<0$ for $\xi$ large (since $Y(\eta)<0$ for $\eta>\eta_0$). Thus, the only way that \eqref{interm7} can hold true is that the first and the last terms in \eqref{interm7} should have the same order as $\xi\to\infty$. But it is well-known that the solutions of a differential equation of the form
$$
y''(x)=x^{\sigma}y^q(x), \quad \sigma>0, \quad 0<q<1,
$$
are necessarily increasing at infinity. Applying this general result for $y=f^m$, $q=p/m$, we reach a contradiction with the negativity of $Y(\xi)$ and thus of $f'(\xi)$ for $\xi$ large, showing that this case is not possible.

$\bullet$ we are left with the case $(XZ)(\eta)$ bounded for $\eta\in(\eta_0,\eta^+)$. In this case, we perform the change of variable $W=XZ$ in the system \eqref{PSsyst} and we are left with the system
\begin{equation}\label{PSsyst3}
\left\{\begin{array}{ll}\dot{X}=X[(m-1)Y-2X],\\
\dot{Y}=-Y^2+\frac{\beta}{\alpha}Y-X-NXY+W,\\
\dot{W}=W[(m+p-2)Y+(\sigma-2)X],\end{array}\right.
\end{equation}
with the same independent variable $\eta$. It then follows from \eqref{interm5} and the assumption of boundedness of $W=XZ$ that the trajectory $r_C$ remains in a compact set in variables $(X,Y,W)$, thus standard results in dynamical systems imply that $\eta^+=\infty$. Moreover, since $XZ$ is bounded along the trajectory $r_C$ and $Z_{\infty}=\infty$, it necessarily follows that $X_{\infty}=0$. Similar arguments as in Step 1, based on \cite[Theorem 2 and Corollary, Section 3.2]{Pe} show that the $\omega$-limit set of the trajectory is contained in the plane $X=0$, which, together with the first equation of the system \eqref{PSsyst}, also proves that $Y=0$. Thus, it follows that the trajectory ends in a part of the $W$-axis of the plane $X=0$ and the convergence to zero of both $X$ and $Y$ in the second equation of the system \eqref{PSsyst3} easily gives that the trajectory ends in the unique critical point contained in the $W$-axis, that is, $(X,Y,W)=(0,0,0)$. Denoting by $Q'$ this critical point, the linearization of the system \eqref{PSsyst3} in a neighborhood of $Q'$ is given by the matrix
$$
M(Q')=\left(
        \begin{array}{ccc}
          0 & 0 & 0 \\
          -1 & \frac{\beta}{\alpha} & 1 \\
          0 & 0 & 0 \\
        \end{array}
      \right),
$$
whose only non-zero eigenvalue is positive. Thus, a trajectory entering $Q'$ has to do it on its (unique) two-dimensional center manifold. An analysis similar to the one performed in the proof of Lemma \ref{lem.Q0} gives that the center manifold of $Q'$ has the form
\begin{equation}\label{cmQp}
\frac{\beta}{\alpha}Y-X+W=aX^2+bXW+cW^2+O(|(X,W)|^3),
\end{equation}
where $a$, $b$ and $c$ can be calculated explicitly (we omit this calculation here). But $W/X=Z\to\infty$ and the boundedness and negativity of $Y$ established in \eqref{interm5} also entail that $W/Y\to-\infty$. Thus, $W$ dominates over the left-hand side of \eqref{cmQp} (as well as above all the terms in the right-hand side, which is obvious) and this leads to a contradiction, as the equality \eqref{cmQp} is not possible. The contradiction stems from the fact that we assumed that $Z(\eta)\to\infty$ as $\eta\to\eta^+$, hence $r_C$ cannot have this behavior.

\medskip

\noindent \textbf{Step 5. If $\eta\mapsto Z(\eta)$ is decreasing we reach the conclusion.} Assume next that there is $\eta_1\in(\eta_0,\eta^+)$ such that the trajectory $r_C$ remains in the region
$$
\mathcal{R}_3:=\{(X,Y,Z)\in\real^3: X\geq0, Z\geq0, (p-1)Y+\sigma X\leq0\}
$$
for any $\eta\in(\eta_1,\eta^+)$. Then $X(\eta)$, $Z(\eta)$ are non-increasing for $\eta\in(\eta_1,\eta^+)$ and thus there exist
\begin{equation}\label{interm6}
X_{\infty}:=\lim\limits_{\eta\to\eta^+}X(\eta)<\infty, \quad Z_{\infty}:=\lim\limits_{\eta\to\infty}Z(\eta)<\infty.
\end{equation}
The second equation of the system \eqref{PSsyst} and the boundedness of both $X$ and $Z$ implies that there is $Y_0>0$ sufficiently large such that $\dot{Y}(\eta)<0$ if $Y(\eta)<-Y_0$. But, in such case, the trajectory $r_C$ would connect to the stable node $Q_3$, which is a contradiction since $C\not\in\mathcal{A}$. Thus, $\eta\mapsto Y(\eta)$ is also bounded for $\eta\in(\eta_1,\eta^+)$ and the trajectory $r_C$ remains forever in a compact set. By similar arguments as in Step 1, we deduce that $X_{\infty}=0$ and $Y(\eta)\to0$ as $\eta\to\eta^+$ as well, and thus the trajectory reaches a critical point of the form $(0,0,Z_{\infty})$. We then infer from Lemma \ref{lem.Qg} that $Z_{\infty}=\gamma_0$, as claimed.

\medskip

\noindent \textbf{Step 6. The oscillating case.} The direction of the flow of the system \eqref{PSsyst} across the plane $(p-1)Y+\sigma X=0$ (where the monotonicity of the $Z$ coordinate changes), with normal direction $(\sigma,p-1,0)$, is given by the sign of the expression
$$
G(X,Z):=(p-1)X\left(Z-1+\frac{m-p}{\sigma+2}\right)+\frac{\sigma[(N-2)(p-1)-m\sigma]}{p-1}X^2.
$$
Assume that our trajectory $r_C$ presents infinitely many oscillations with respect to the plane $(p-1)Y+\sigma X=0$ (if this is not the case, either one of the steps 3 or 5 applies to it starting from some $\eta_1>\eta_0$ sufficiently large). Since $X(\eta)<X(\eta_0)$ is bounded along this trajectory, it follows that $G(X,Z)>0$ for $Z$ sufficiently large, yielding that $\eta\mapsto Z(\eta)$ is also bounded for $\eta\in(\eta_0,\eta^+)$ along the trajectory $r_C$, implying that $\eta^+=\infty$. Since $X(\eta)$ is decreasing, there is $X_{\infty}$ as in \eqref{interm6} and arguments as in Step 1 (in the plane $X=X_{\infty}$ instead of $Z=Z_{\infty}$) readily show that $X_{\infty}=0$. Indeed, if $X_{\infty}>0$, then the $\omega$-limit set of $r_C$ is a trajectory included in the plane $X=X_{\infty}$ and the first equation of the system \eqref{PSsyst} ensures that $Y=2X_{\infty}/(m-1)$ is also constant. By replacing $\dot{Y}=0$ in the second equation of \eqref{PSsyst} we derive that $Z$ is also constant, thus the $\omega$-limit set is a finite critical point lying in a plane $X=X_{\infty}>0$, and this is a contradiction. 

We next show that $Y(\eta)\to0$ as $\eta\to\infty$. Assume that there exist sequences $\{\eta_{j}^{m}\}_{j\geq1}$ of points of minima for $\eta\mapsto Y(\eta)$ and $\{\eta_{j}^{M}\}_{j\geq1}$ of points of maxima for $\eta\mapsto Y(\eta)$ such that $\eta_{j}^m\to\infty$, $\eta_j^M\to\infty$ as $j\to\infty$ (otherwise $Y(\eta)$ becomes monotone for $\eta$ large and thus convergent). Then, by taking into account that $\dot{Y}(\eta_j^m)=\dot{Y}(\eta_j^M)=0$ and evaluating the second equation of the system \eqref{PSsyst} at $\eta_j^m$, respectively $\eta_j^M$, we readily infer from the fact that $X(\eta)\to0$ and the boundedness of $Y(\eta)$ and $Z(\eta)$ for $\eta$ large that 
$$
\lim\limits_{j\to\infty}\left[-Y^2(\eta_j^m)+\frac{\beta}{\alpha}Y(\eta_j^m)\right]
=\lim\limits_{j\to\infty}\left[-Y^2(\eta_j^M)+\frac{\beta}{\alpha}Y(\eta_j^M)\right]=0.
$$
These limits and the negativity of $Y(\eta)$ for $\eta>\eta_0$ established in \eqref{interm5} give that $Y(\eta_j^m)$ and $Y(\eta_j^M)$ both tend to zero as $j\to\infty$ and thus $Y(\eta)\to0$ as $\eta\to\infty$. The same holds true if $Y(\eta)$ is monotone for $\eta>\eta_*$ large, as it follows immediately from the assumption of oscillations with respect to the plane $(p-1)Y+\sigma X=0$ and the fact that $X(\eta)\to0$. Thus, 
$$
\lim\limits_{\eta\to\infty}X(\eta)=\lim\limits_{\eta\to\infty}Y(\eta)=0
$$
and we are left to show that the trajectory $r_C$ enters the critical point $Q_{\gamma_0}$.

\medskip 

\noindent \textbf{Step 7. End of the proof.} Assume for contradiction that we are under the assumptions of Step 6 and $r_C$ does not connect to $Q_{\gamma_0}$. Thus, by Lemma \ref{lem.Qg}, the only alternative left after the analysis performed in Steps 4-5 is an infinite (but bounded) oscillation approaching the critical line $X=Y=0$. Let $g(\xi):=\xi^{-\sigma/(p-1)}f(\xi)$. Direct calculations starting from \eqref{SSODE} show that the equation solved by $g$ is 
\begin{equation}\label{ODEg}
\begin{split}
\xi^2(g^m)''(\xi)&-\left(\frac{2m\sigma}{p-1}-N+1\right)\xi(g^m)'(\xi)+\frac{m\sigma}{p-1}\left(\frac{m\sigma}{p-1}-N+2\right)g^m(\xi)\\
&+\xi^{[\sigma(m-1)+2(p-1)]/(p-1)}\left[-\beta\xi g'(\xi)+\frac{1}{p-1}g(\xi)-g^p(\xi)\right]=0.
\end{split}
\end{equation}
Assume first that $m\sigma/(p-1)-N+2<0$. Then, by evaluating \eqref{ODEg} at any point of maximum $\xi_{M}$ of $g$, we get that 
$$
\frac{1}{p-1}g(\xi_M)-g^p(\xi_M)>0, \quad {\rm that \ is}, \quad g(\xi_M)<\left(\frac{1}{p-1}\right)^{1/(p-1)}=c^*.
$$
Hence

$\bullet$ either $g$ is monotone increasing (since $g(0)=0$) and bounded, hence convergent as $\xi\to\infty$, which implies that 
$$
Z(\eta)=\frac{1}{\alpha}g(\xi)^{p-1}
$$
is also convergent and its limit has to be $\gamma_0$, by Lemma \ref{lem.Qg}.

$\bullet$ or $g$ has infinitely many points of maxima and minima and thus $g(\xi)<C(p)$ for any $\xi\in(0,\infty)$. Going back to the phase space, this means that the trajectory $r_C$ has as $\omega$-limit points a subset of the segment $(0,\gamma_0)$ of the $Z$-axis. But this contradicts the assumption of infinitely many oscillations with respect to the plane $(p-1)Y+\sigma X=0$, since the center manifolds of the points $Q_{\gamma}$ with $\gamma<\gamma_0$, given in a first approximation by 
$$
Y=\frac{(\sigma+2)(1-\gamma)}{m-p}X,
$$
lie all them on the same side of the plane $(p-1)Y+\sigma X=0$. This contradiction implies that this case is not possible.

Assume now that $m\sigma/(p-1)-N+2>0$. Then, by evaluating \eqref{ODEg} at any point of minimum $\xi_{m}$ of $g$, we get that
$$
\frac{1}{p-1}g(\xi_m)-g^p(\xi_m)<0, \quad {\rm that \ is}, \quad g(\xi_m)>c^*.
$$
Thus, by similar arguments as above, the $\omega$-limit set of the trajectory $r_C$ is in this case either a point (which can only be $Q_{\gamma_0}$) or a segment of the critical $Z$-axis contained in $(\gamma_0,\infty)$, and again we reach a contradiction with the assumption of infinitely many oscillations with respect to the plane $(p-1)Y+\sigma X=0$. This completes the proof by establishing that $r_C$ connects to $Q_{\gamma_0}$ for any $C\in\mathcal{B}$.
\end{proof}

\section{Proof of Theorem \ref{th.SSSdc}: uniqueness}

The proof of the uniqueness of a solution as in Theorem \ref{th.SSSdc} does no longer employ the phase space, but, instead, it is performed directly on the profiles solving Eq. \eqref{SSODE} and corresponding self-similar solutions. In a first step, it is shown that, if there are two solutions with dead-core as stated in Theorem \ref{th.SSSdc}, then their profiles are ordered. In a second step, we derive a contradiction by a technique of scaling and optimal parameters. 
\begin{proof}[Proof of Theorem \ref{th.SSSdc}: uniqueness]
Assume for contradiction that there are two self-similar solutions in the form \eqref{SSS} such that their profiles $f_1$ and $f_2$ solve \eqref{SSODE} and satisfy \eqref{beh.dc} and \eqref{beh.inf} at the same time. Denote by $\xi_1\in(0,\infty)$ and $\xi_2\in(0,\infty)$ the left edges of their supports corresponding to the behavior \eqref{beh.dc} and assume (without loss of generality) that $\xi_1<\xi_2$. We split the rest of the proof into several steps to make it easier to follow.

\medskip 

\noindent \textbf{Step 1. At most one point of intersection.} In a first step, we prove that the profiles $f_1$ and $f_2$ have at most a single point of intersection, that is, there is at most one point $\overline{\xi}\in(\xi_2,\infty)$ such that
$$
f_1(\xi)>f_2(\xi), \ {\rm if} \ \xi\in(\xi_1,\overline{\xi}), \quad f_1(\xi)<f_2(\xi) \ {\rm if} \ \xi\in(\overline{\xi},\infty).
$$
The proof of this fact adapts an argument from \cite{CV96}. Let us construct the self-similar solutions with profiles $f_1$ and $f_2$, that is, $u_i(x,t)=t^{-\alpha}f_i(|x|t^{\beta})$, and, for $\tau>0$, introduce
$$
u_{2,\tau}(x,t)=u_2(x,t+\tau-1),
$$
which is a solution to Eq. \eqref{eq1} as well, for $t>1-\tau$. Let us then observe that 
\begin{equation*}
\begin{split}
t^{\alpha}u_{2,\tau}(x,t)&=t^{\alpha}(t+\tau-1)^{-\alpha}f_2(|x|(t+\tau-1)^{\beta})\\
&=\left(1+\frac{\tau-1}{t}\right)^{-\alpha}f_2\left(|x|t^{\beta}\left(1+\frac{\tau-1}{t}\right)^{\beta}\right)\\
&=\left(1+\frac{\tau-1}{t}\right)^{-\alpha}f_2\left(\xi\left(1+\frac{\tau-1}{t}\right)^{\beta}\right)
\end{split}
\end{equation*}
whence 
\begin{equation}\label{interm8}
\lim\limits_{t\to\infty}t^{\alpha}u_{2,\tau}(x,t)=f_2(\xi)
\end{equation}
for any $\tau>0$. The idea of this step is to show that there exists $\tau$ sufficiently large such that the number of intersections of 
$$
u_1(x,1)=f_1(\xi)\quad {\rm and} \quad u_{2,\tau}(x,1)=\tau^{-\alpha}f_2(\xi\tau^{\beta})
$$ 
is exactly one. Note that the support of $u_{2,\tau}(\cdot,1)$ is $(\tau^{-\beta}\xi_2,\infty)$, hence it tends to cover the whole interval $(0,\infty)$ as $\tau\to\infty$. Moreover, taking into account the behavior \eqref{beh.inf} of $f_2$ as $\xi\to\infty$, for a fixed $\xi\in(0,\infty)$ we have 
\begin{equation}\label{interm9}
\tau^{-\alpha}f_2(\xi\tau^{\beta})\sim\tau^{-\alpha}C_*(\xi\tau^{\beta})^{-\sigma/(p-1)}=C_*\tau^{-1/(p-1)}\xi^{-\sigma/(p-1)},
\end{equation}
as $\tau\to\infty$. Let us next denote by $\xi_M\in(\xi_1,\infty)$ the unique maximum point of the profile $f_1$, that is, $f_1$ is increasing on $(\xi_1,\xi_M)$ and decreasing on $(\xi_M,\xi_{\infty})$. On the one hand, it is then obvious that there is $\tau_0>0$ sufficiently large such that for any $\tau>\tau_0$ there is a unique intersection point between the graph of the decreasing function derived in \eqref{interm9}
$$
\xi\mapsto h_{\tau}(\xi):=c_*\tau^{-1/(p-1)}\xi^{-\sigma/(p-1)}
$$ 
and the increasing part of the profile $f_1$ for $\xi\in(\xi_1,\xi_M)$. On the other hand, observing that 
$$
\lim\limits_{\xi\to\infty}\xi^{\sigma/(p-1)}f_1(\xi)=c_*, \quad \lim\limits_{\xi\to\infty}\xi^{\sigma/(p-1)}h_{\tau}(\xi)=c_*\tau^{-1/(p-1)},
$$
we can choose $\tau$ sufficiently large such that 
\begin{equation}\label{interm10}
c_*\tau^{-1/(p-1)}<\inf\left\{\xi^{\sigma/(p-1)}f_1(\xi): \xi\in(\xi_M,\infty)\right\}.
\end{equation}
If $\tau$ is chosen such that \eqref{interm10} is in force and also $\tau>\tau_0$, then $f_1(\xi)>h_{\tau}(\xi)$ for any $\xi\in(\xi_M,\infty)$ and thus the profiles of $u_1$ and $u_{2,\tau}$ have a single point of intersection at $t=1$. Since both are solutions, a standard argument of intersection comparison ensures that there is at most one intersection point between the profiles of $u_1(t)$ and $u_{2,\tau}(t)$ at any $t>1$. Letting $t\to\infty$ and employing \eqref{interm8}, we deduce the same property for $f_1$ and $f_2$, as stated.

\medskip 

\noindent \textbf{Step 2. $f_1$ and $f_2$ are ordered.} Consider again the solutions $u_1$ and $u_{2,\tau}$ with $\tau>0$ sufficiently large as in Step 1. Their interface points at time $t>0$ are given by 
$$
x_1(t):=\xi_1t^{-\beta}, \quad {\rm respectively} \quad x_{2,\tau}(t)=\xi_2(t+\tau-1)^{-\beta},
$$ 
hence there is 
\begin{equation}\label{deftstar}
t_*:=\frac{\tau-1}{\left(\frac{\xi_2}{\xi_1}\right)^{1/\beta}-1}
\end{equation}
such that $x_1(t_*)=x_{2,\tau}(t_*)$. We next compare the local behavior near the common interface point of the solutions $u_1(x,t_*)$ and $u_{2,\tau}(x,t_*)$. To this end, we deduce on the one hand from \eqref{beh.interf} applied to $f_1$ that 
\begin{equation}\label{exp1}
\begin{split}
u_1^{m-1}(x,t_*)&\sim(t_*)^{-\alpha(m-1)}\frac{\beta(m-1)}{m}\xi_1(|x|t_*^{\beta}-\xi_1)\\
&=\frac{\beta(m-1)}{m}t_*^{-\alpha(m-1)+\beta}\xi_1(|x|-x_1(t_*))
\end{split}
\end{equation}
and on the other hand, from \eqref{beh.interf} applied to $f_2$, that 
\begin{equation}\label{exp2}
\begin{split}
u_{2,\tau}^{m-1}(x,t_*)&\sim(t_*+\tau-1)^{-\alpha(m-1)}\frac{\beta(m-1)}{m}\xi_2(|x|(t_*+\tau-1)^{\beta}-\xi_2)\\
&=\frac{\beta(m-1)}{m}(t_*+\tau-1)^{-\alpha(m-1)+\beta}\xi_2(|x|-x_{2,\tau}(t_*)).
\end{split}
\end{equation}
We next compare in a neighborhood of $x_1(t_*)=x_{2,\tau}(t_*)$ the terms in the right-hand side of \eqref{exp1} and \eqref{exp2}. Recalling the definition of $t_*$ from \eqref{deftstar}, we note that
\begin{equation*}
\begin{split}
\frac{(t_*+\tau-1)^{-\alpha(m-1)+\beta}\xi_2}{t_*^{-\alpha(m-1)+\beta}\xi_1}&=\left(1+\frac{\tau-1}{t_*}\right)^{\beta-\alpha(m-1)}\frac{\xi_2}{\xi_1}\\
&=\left(\frac{\xi_2}{\xi_1}\right)^{2-\alpha(m-1)/\beta}=\left(\frac{\xi_2}{\xi_1}\right)^{-1/\beta}<1,
\end{split}
\end{equation*}
since we are under the assumption that $\xi_2>\xi_1$. It thus follows that there is $\epsilon>0$ such that 
$$
u_1(x,t_*)>u_{2,\tau}(x,t_*), \quad |x|\in(\xi_1t_*^{-\beta},\xi_1t_*^{-\beta}+\epsilon).
$$
Since the same order is also fulfilled as $|x|\to\infty$, it follows that either $u_1(x,t_*)\geq u_{2,\tau}(x,t_*)$ for any $x$ such that $|x|>x_1(t_*)=x_{2,\tau}(t_*)$, or the profiles of $u_1(t_*)$ and $u_{2,\tau}(t_*)$ have at least two intersection points. But the latter is a contradiction to the result of Step 1, hence $u_1$ and $u_{2,\tau}$ are ordered at $t=t_*$. The comparison principle then ensures that $u_1(t)\geq u_{2,\tau}(t)$ for any $t>t_*$. By multiplying by $t^{\alpha}$ and passing to the limit as $t\to\infty$, we readily infer that $f_1(\xi)\geq f_2(\xi)$ for any $\xi\in(0,\infty)$. 

\medskip

\noindent \textbf{Step 3. Uniqueness.} We adapt an idea stemming (up to our knowledge) from \cite{YeYin} and that the authors employed with success in their previous papers \cite{IM25, IM26}. Let us introduce the following scaled functions 
$$
f_\lambda(\xi):=\lambda^{-2/(m-1)}f_2(\lambda\xi).
$$ 
It is straightforward to check (starting from \eqref{SSODE} solved by $f_2$) that $f_{\lambda}$ is a solution to the differential equation 
\begin{equation}\label{SSODEresc}
(f_{\lambda}^m)''(\xi)+\frac{N-1}{\xi}(f_{\lambda}^m)'(\xi)+\alpha f_{\lambda}(\xi)-\beta\xi f_{\lambda}'(\xi)-\lambda^{L/(m-1)}\xi^{\sigma}f_{\lambda}^p(\xi)=0,
\end{equation}
where $L=\sigma(m-1)+2(p-1)<0$. Moreover, the left edge of the support of $f_{\lambda}$ is given by $\xi_{\lambda}=\xi_2/\lambda$, while the behavior at infinity, derived from \eqref{beh.inf}, is given by 
$$
f_{\lambda}(\xi)\sim c_*\lambda^{-2/(m-1)-\sigma/(p-1)}\xi^{-\sigma/(p-1)}=\lambda^{-L/[(m-1)(p-1)]}c_*\xi^{-\sigma/(p-1)}.
$$
It follows that, for any $\lambda<1$, we have $f_{\lambda}(\xi)<f_1(\xi)$ in a neighborhood of infinity, and this neighborhood is increasing with $\lambda$ according to the previous equivalence. Assume next for contradiction that there is $\lambda>1$ such that the function $f_{\lambda}$ touches from below $f_1$ at a point $\tilde{\xi}\in(\xi_1,\infty)$, that is 
$$
f_{\lambda}(\tilde{\xi})=f_1(\tilde{\xi}), \quad f_{\lambda}'(\tilde{\xi})=f_1'(\tilde{\xi}), \quad 
(f_{\lambda}^m)''(\tilde{\xi})\leq(f_1^m)''(\tilde{\xi}).
$$
We then obtain from \eqref{SSODEresc} evaluated at $\tilde{\xi}$ and the previous equalities and inequality that 
$$
(f_{1}^m)''(\tilde{\xi})+\frac{N-1}{\tilde{\xi}}(f_{1}^m)'(\tilde{\xi})+\alpha f_{1}(\tilde{\xi})-\beta\tilde{\xi} f_{1}'(\tilde{\xi})-\lambda^{L/(m-1)}\tilde{\xi}^{\sigma}f_{1}^p(\tilde{\xi})\geq0.
$$
By subtracting \eqref{SSODE} applied to $f_1$ and evaluated at $\tilde{\xi}$, we further deduce that 
$$
(\lambda^{L/(m-1)}-1)\tilde{\xi}^{\sigma}f_{1}^p(\tilde{\xi})\leq0,
$$
which is a contradiction since $\lambda>1$ and $L>0$. Thus no such touching from below is possible and, consequently, if we increase $\lambda$, the first parameter for which the profiles of $f_1$ and $f_{\lambda}$ touch is 
$$
\lambda_0=\frac{\xi_2}{\xi_1}>1,
$$
when the two profiles touch exactly at their common left limit of the support $\xi_1$. We still have $f_{\lambda_0}(\xi)<f_1(\xi)$ for $\xi\in(\xi_1,\infty)$. Let us next consider the self-similar functions 
$$
u_1(x,t)=t^{-\alpha}f_1(|x|t^{\beta}), \quad u_{\lambda}(x,t)=t^{-\alpha}f_{\lambda}(|x|t^{\beta}), \quad \lambda>1,
$$
and remark that the latter considerations imply $u_1(x,1)\geq u_{\lambda_0}(x,1)$, for any $x\in\real^N$ such that $|x|\geq\xi_1$. The strict separation also implies that there is $\delta>0$ sufficiently small such that 
\begin{equation}\label{interm12}
u_{\lambda_0}(x,1)<u_1(x,1+\delta), \quad |x|\geq\xi_1.
\end{equation}
By a similar argument as above, we can then adjust the optimal parameter $\lambda_0$ and deduce further from \eqref{interm12} that 
\begin{equation}\label{interm13}
u_{\lambda_1}(x,1)\leq u_1(x,1+\delta), \quad \lambda_1:=\frac{\xi_2(1+\delta)^{\beta}}{\xi_1}>\lambda_0,
\end{equation}
where $\lambda_1$ is again chosen in order to make coincide the supports of $u_{\lambda_1}(\cdot,1)$ and $u_1(\cdot,1+\delta)$. We easily obtain by direct calculation (for a proof where the calculations are written, we refer the reader to \cite[Lemma 5.3]{IM25}) that, for any $\lambda>1$, $u_{\lambda}$ is a subsolution to Eq. \eqref{eq1}. The comparison principle and \eqref{interm13} then ensure that 
$$
u_{\lambda_1}(x,t)\leq u_1(x,t+\delta), \quad t>1,
$$
which, after easy algebraic manipulations, is equivalent to 
\begin{equation}\label{interm14}
\lambda_1^{-2/(m-1)}f_2(\lambda_1\xi)\leq\left(\frac{t+\delta}{t}\right)^{-\alpha}f_1\left(\xi\left(\frac{t+\delta}{t}\right)^{\beta}\right),
\end{equation}
for any $t>1$. Letting $t\to\infty$ in \eqref{interm14}, we find 
$$
f_{\lambda_1}(\xi)=\lambda_1^{-2/(m-1)}f_2(\lambda_1\xi)\leq f_1(\xi), \quad \xi\in(0,\infty),
$$
which is a contradiction, since the supports are ordered in the opposite way. Indeed,
$$
{\rm supp}\,f_1=(\xi_1,\infty)\subset\left(\frac{\xi_1}{(1+\delta)^{\beta}},\infty\right)
=\left(\frac{\xi_2}{\lambda_1},\infty\right)={\rm supp}\,f_{\lambda_1}.
$$
The previous contradiction completes the proof of the uniqueness.
\end{proof}
Let us notice here that the argument employed in the last step of the previous proof departs strongly from the argument of the corresponding proof for $\sigma=0$ in \cite{CV96}, since the argument therein uses in an essential manner the invariance to translations of the equation with $\sigma=0$, a property that is no longer true for Eq. \eqref{eq1}.
 
\section{Proof of Theorem \ref{th.SSSpos}}

The existence part of Theorem \ref{th.SSSpos} is equivalent to show that there exists a trajectory of the system \eqref{PSsyst} (understood together with its Poincar\'e compactification giving rise to the critical point $P_0$) between $P_0$ and $Q_{\gamma_0}$. The strategy of the proof is to perform a shooting from $P_0$ on the trajectories $l_C$ given in Lemma \ref{lem.P0}.
\begin{proof}[Proof of Theorem \ref{th.SSSpos}]
Consider the following three sets, analogous to the ones in \eqref{sets}:
\begin{equation*}
\begin{split}
&\mathcal{E}:=\{C\in(0,\infty): {\rm the \ trajectory} \ l_C \ {\rm connects \ to} \ Q_3\},\\
&\mathcal{G}:=\{C\in(0,\infty): {\rm the \ trajectory} \ l_C \ {\rm enters \ the \ region} \ \mathcal{R}\},\\
&\mathcal{F}:=(0,\infty)\setminus(\mathcal{E}\cup\mathcal{G}).
\end{split}
\end{equation*}
On the one hand, Proposition \ref{prop.Z0} shows that the trajectory $l_0$ connects to the stable node $Q_3$. The stability of $Q_3$ then implies that there is $C_*>0$ such that $(0,C_*)\subseteq\mathcal{E}$. Moreover, the same stability of $Q_3$ also gives that $\mathcal{E}$ is an open set. On the other hand, the unique trajectory going out of $P_0$ in the invariant plane $x=0$ of the system \eqref{PSsyst2}, that we will call $l_{\infty}$ for short, is tangent to the eigenvector $v_x=(0,1,N+\sigma)$ and thus enters the half-plane $y>0$ of the phase plane associated to the reduced system
\begin{equation}\label{PSsyst2x0}
\left\{\begin{array}{ll}\frac{dy}{d\theta}=-(N-2)y+z-my^2, \\ \frac{dz}{d\theta}=z(\sigma+2+(p-m)y).\end{array}\right.
\end{equation}
Since the direction of the flow of the system \eqref{PSsyst2x0} across the line $y=0$ points towards the positive half-plane, it follows that the trajectory $l_{\infty}$ remains forever in the half-plane $y>0$. Moreover, the condition $Z>1$ of the positively invariant region $\mathcal{R}$ translates into $z>x$, a condition directly fulfilled in the half-plane $x=0, y>0$. Thus, the trajectory $l_{\infty}$ enters the region $\mathcal{R}$. Since the second equivalence in \eqref{lC} can be written equivalently as 
$$
x(\theta)\sim\overline{C}^{2/(\sigma+2)}z(\theta)^{2/(\sigma+2)}, \quad \overline{C}=\frac{1}{C},
$$
we can see that $l_{\infty}$ corresponds to $\overline{C}=0$ and, by continuity, there is $C^*>C_*>0$ such that $(C^*,\infty)\subseteq\mathcal{G}$. It thus follows that $\mathcal{G}$ is non-empty and open as well. We thus infer that $\mathcal{F}$ is non-empty. Pick $C\in\mathcal{F}$. Then, the first equivalence in \eqref{lC} proves that the trajectory $l_C$ with $C\in\mathcal{F}$ goes out into the half-space $y<0$ in variables $(x,y,z)$ of the system \eqref{PSsyst2}, or, equivalently, in the half-space $Y<0$ in the initial variables of the system \eqref{PSsyst}. As shown in Step 2 in Section \ref{sec.existQ1}, the trajectory cannot cross again the plane $Y=0$ since it would enter directly the region $\mathcal{R}$, contradicting the fact that $C\in\mathcal{F}$. Thus, the trajectory $l_C$ remains forever in the half-space $Y<0$, which entails the monotonicity of $X(\eta)$. The rest of the proof is a complete repetition of Steps 3-7 of the previous section \ref{sec.existQ1}, as exactly the same arguments apply to the trajectory $l_C$ and drive it to the critical point $Q_{\gamma_0}$. 

The proof of the uniqueness in Theorem \ref{th.SSSpos} is completely identical to the similar proof in the range $p>m$, and we refer the reader to the proof of Theorem 2.1 in \cite[Section 4]{IM26}, as a simple inspection of the proof therein shows that the order between $p$ and $m$ is completely irrelevant for this proof.
\end{proof}

\section{Proof of Theorem \ref{th.SSSzero}}

We are only left to prove Theorem \ref{th.SSSzero}. We first notice that a self-similar solution satisfying the local behavior \eqref{beh.zero} as $\xi\to0$ for some $K>0$ is seen in the alternative formulation \eqref{PSchange} as a trajectory of the system \eqref{PSsyst} going out of the critical point $Q_0$ with 
$$
\lim\limits_{\eta\to-\infty}\frac{Z(\eta)}{X(\eta)}=\frac{1}{m}K^{p-m}.
$$
We thus perform the change of variable $Z=VX$ in the system \eqref{PSsyst}. This change leads to a kind of ``blow-up" of the critical point $Q_0$, allowing to identify the trajectories with a more precise behavior.

A direct calculation shows that the system \eqref{PSsyst} transforms in the new variables $(X,Y,V)$ in 
\begin{equation}\label{PSsystW}
	\left\{\begin{array}{ll}\dot{X}=X[(m-1)Y-2X],\\
		\dot{Y}=-Y^2+\frac{\beta}{\alpha}Y-X-NXY+X^2V,\\
		\dot{V}=V[(p-m)Y+(\sigma+2)X],\end{array}\right.
\end{equation}
and we are interested in the unstable or center-unstable manifolds of the critical points $(0,0,v_0)$ for $v_0\in(0,\infty)$. Let us introduce the constant 
\begin{equation*}
	v_*:=\frac{[m(N+\sigma)-p(N-2)](\sigma+2)}{(m-p)^2}.
\end{equation*}
The local analysis of the critical point $Q(v_0)=(0,0,v_0)$ in the system \eqref{PSsystW} is given in the following result. 
\begin{lemma}\label{lem.v0}
The linearization of the system \eqref{PSsystW} in a neighborhood of the critical point $Q(v_0)=(0,0,v_0)$ with $v_0\in(0,\infty)$ has a two-dimensional center manifold and a one-dimensional unstable manifold. The flow on the center manifold goes in the unstable direction for $v_0\in(v_*,\infty)$ and behaves as a saddle point for $v_0\in(0,v_*)$. Moreover, for $v_0=v_*$, the following line 
\begin{equation}\label{sol.line}
(\sigma+2)X+(p-m)Y=0, \quad V=v_*,
\end{equation}
is an explicit trajectory of the system \eqref{PSsystW}.
\end{lemma}	
\begin{proof}
The linearization of the system \eqref{PSsystW} in a neighborhood of the critical point $Q(v_0)$ has the matrix
$$
M(v_0)=\left(
\begin{array}{ccc}
	0 & 0 & 0 \\
	-1 & \frac{\beta}{\alpha} & 0 \\
	(\sigma+2)v_0 & (p-m)v_0 & 0 \\
\end{array}
\right),
$$
with a positive eigenvalue and a zero eigenvalue of multiplicity two. It thus follows that the system in a neighborhood of $Q(v_0)$ has a one-dimensional unstable manifold and two-dimensional center manifolds. In order to analyze the center manifolds, we proceed as in the proof of Lemma \ref{lem.Q0} by introducing the variable
$$
T:=\frac{\beta}{\alpha}Y-X, \quad {\rm that \ is}, \quad Y=\frac{\alpha}{\beta}(T+X).
$$
By letting, as in \cite[Theorem 3, Section 2.5]{Carr},
$$
T=aX^2+bXV+cV^2+o(|(X,V)|^2),
$$
and calculating the flow of the system \eqref{PSsystW} on the previous surface, we deduce that
$$
a=\frac{[m(N+\sigma)-p(N-2)](\sigma+2)-(m-p)^2v_0}{(m-p)(\sigma+2)}, \quad b=c=0,
$$
and thus the equation of the center manifold of $Q(w_0)$ becomes 
$$
Y=\frac{\alpha}{\beta}X+\frac{a\alpha}{\beta}X^2+o(|(X,V)|^2)=\frac{\alpha}{\beta}X+(v_*-v_0)X^2+o(|(X,V)|^2).
$$
Let us observe at this point that $a>0$ if $v_0\in(0,v_*)$ and $a<0$ if $v_0>v_*$. We now apply \cite[Theorem 2,Section 2.4]{Carr} in order to study the flow on the center manifold of $Q(v_0)$, which is given by the reduced system obtained by replacing the expression of the center manifold above in the first and third equations of the system \eqref{PSsystW} and keeping only the quadratic terms. This reduced system reads
\begin{equation*}
	\left\{\begin{array}{ll}\dot{X}=\frac{1}{\beta}X^2+O(|(X,V)|^3),\\[1mm]
	\dot{V}=-av_0(\sigma+2)X^2+O(|(X,V)|^3).\end{array}\right.
\end{equation*}
Taking into account that $a<0$ for $v_0>v_*$, we infer that the center manifold has unstable flow if $v_0>v_*$ and thus the critical point $Q(v_0)$ has a three-dimensional center-unstable manifold in this range. On the contrary, if $v_0\in(0,v_*)$, the flow on the center manifold is similar to a saddle point. Finally, letting $v_0=v_*$, we observe that the line \eqref{sol.line} corresponds to the unbounded stationary solution given in \eqref{stat.sol} and is thus a trajectory of the system \eqref{PSsystW}, completing the proof. 
\end{proof}
We need one more preparatory result, concerning the invariant plane $X=0$, where the system \eqref{PSsystW} reduces to
\begin{equation}\label{PSsystWX0}
	\left\{\begin{array}{ll}
		\dot{Y}=-Y^2+\frac{\beta}{\alpha}Y,\\
		\dot{V}=-(m-p)YV,\end{array}\right.
\end{equation}
\begin{lemma}\label{lem.VX0}
The critical point $Q_1=(\beta/\alpha,0)$ in the system \eqref{PSsystWX0} is a stable node. For any $v_0\in(0,\infty)$, there exists a trajectory of the system \eqref{PSsystWX0} connecting the critical point $Q(v_0)=(0,v_0)$ to $Q_1$.
\end{lemma}
\begin{proof}
The linearization of the system \eqref{PSsystWX0} in a neighborhood of $Q_1$ has the matrix 
$$
M_1=\begin{pmatrix}
	-\frac{\beta}{\alpha} & 0 \\
	0 & -\frac{(m-p)\beta}{\alpha}
\end{pmatrix},
$$
with two negative eigenvalues. Moreover, by directly integrating the system \eqref{PSsystWX0}, we obtain the explicit family of trajectories
$$
Y(\eta)=\frac{\beta}{\alpha(1+Ae^{-\beta\eta/\alpha})}, \qquad V(\eta)=C\left(e^{\beta\eta/\alpha}+A\right)^{-(m-p)\alpha/\beta},
$$
where $A$ and $C>0$ are integration constants. We observe that 
$$
\lim\limits_{\eta\to\infty}Y(\eta)=\frac{\beta}{\alpha}, \qquad \lim\limits_{\eta\to\infty}V(\eta)=0,
$$
whence the trajectory enters $Q_1$, while
$$
\lim\limits_{\eta\to-\infty}Y(\eta)=0, \qquad \lim\limits_{\eta\to\infty}V(\eta)=CA^{-(m-p)\alpha/\beta}.
$$
It is then sufficient to pick $A$ and $C$ such that 
$$
CA^{-(m-p)\alpha/\beta}=v_0
$$
to obtain a trajectory connecting $Q(v_0)$ to $Q_1$, as claimed. Let us note also that the previous trajectory can be written, in $(Y,V)$ variables, as the following curve:
$$
Y=\frac{\beta}{\alpha}\left(1-\left(\frac{V}{v_0}\right)^{1/(m-p)}\right).
$$ 
\end{proof}
We are now in a position to prove Theorem \ref{th.SSSzero}.
\begin{proof}[Proof of Theorem \ref{th.SSSzero}]
The proof is divided into three steps. In the first and third ones we employ the phase space, but in the second step we work directly with the equation \eqref{eq1}.

\medskip 

\noindent \textbf{Step 1.} Let first $v_0=v_*$. On the one hand, we infer from Lemma \ref{lem.v0} that there is a trajectory (the explicit trajectory \eqref{sol.line}) stemming from $Q(v_*)$ and entering the invariant region $\mathcal{R}$ introduced in \eqref{invreg} (which in the $(X,V)$ variables reads $XV>1$, $Y>0$). An argument of continuity implies that there exists $\delta\in(0,v_*)$ such that, for any $v_0\in(v_*-\delta,v_*+\delta)$ there is a trajectory contained in the center-unstable manifold of $Q(v_0)$ and entering the invariant region $\mathcal{R}$.

On the other hand, for any $v_0\in(0,\infty)$, Lemma \ref{lem.VX0} guarantees that there exists a trajectory connecting $Q(v_0)$ to $Q_1$ contained in the plane $X=0$. Moreover, the invariant plane $V=0$ in the system \eqref{PSsystW} is identical to the plane $Z=0$ in the system \eqref{PSsyst} analyzed in Proposition \ref{prop.Z0} and thus the unique trajectory stemming from $Q_1$ contained in the plane $V=0$ connects to $Q_3$ (as shown in Proposition \ref{prop.Z0}). Since $Q_1$ is a (hyperbolic) saddle point, the local behavior in a neighborhood of a saddle (see for example \cite[Theorem 2.9]{Shilnikov}) ensures that the trajectories contained in the center-unstable manifold of $Q(v_0)$ and in a tubular neighborhood of the unique trajectory connecting $Q(v_0)$ to $Q_1$ will follow, in a neighborhood of $Q_1$, the direction of the trajectory going out of $Q_1$ and thus connect to $Q_3$ as well.

Picking now $v_0\in(v_*-\delta,v_*+\delta)$, undoing the change of variable $Z=VX$ and redefining the sets $\mathcal{A}$, $\mathcal{B}$ and $\mathcal{C}$ as in \eqref{sets} (but related to the trajectories on the center-unstable manifold of $Q(v_0)$), we can repeat the proof of Theorem \ref{th.SSSdc} in order to conclude that there is a trajectory connecting $Q(v_0)$ to $Q_{\gamma_0}$. The proof is complete for any $v_0\in(v_*-\delta,v_*+\delta)$. 

In particular, for $v_0=v_*$, by undoing the change of variable \eqref{PSchange}, we have just shown that there exists at least a profile $f_{K_0}$ solution to \eqref{SSODE} and satisfying the local behaviors \eqref{beh.zero} with $K=K_0$ given in \eqref{stat.sol} as $\xi\to0$ and \eqref{beh.inf} as $\xi\to\infty$, and thus a self-similar solution 
$$
U_{K_0}(x,t)=t^{-\alpha}f_*(|x|t^{\beta})
$$
to Eq. \eqref{eq1}. Note that the solution $U_{K_0}$ as as initial trace (that is, as $t\to0$) the function 
$$
U_{0,K_0}(x)=K_0|x|^{(\sigma+2)/(m-p)}, \quad x\in\real^N.
$$

\medskip 

\noindent \textbf{Step 2.} Fix now $K\in(0,K_0)$. In this case, the construction of a self-similar solution follows a standard approximation argument by truncation. For $n\geq1$ natural number, let us consider the solutions $U_{n,K}$ of the following Cauchy problem:
\begin{equation}\label{cpKn}
\left\{\begin{array}{ll}u_t=\Delta u^m-|x|^{\sigma}u^p, & (x,t)\in \mathbb{R}^N\times(0,\infty),\\
u(x,0)=\min\{K|x|^{(\sigma+2)/(m-p)},n\}, & x\in \mathbb{R^N}.
\end{array}\right.
\end{equation}
Since $K<K_0$, the comparison principle for bounded solutions to Eq. \eqref{eq1} (see, for example, \cite[Theorem 1.1]{ILS24} and its proof in \cite[Section 2]{ILS24}) ensures that 
$$
U_{n,K}(x,t)\leq U_{n+1,K}(x,t)\leq U_{K_0}(x,t),
$$
for any $n\geq1$ and $(x,t)\in \mathbb{R}^N\times(0,\infty)$. This implies that there is 
$$
U_K(x,t):=\lim\limits_{n\to\infty}U_{n,K}(x,t)\leq U_{K_0}(x,t), \quad (x,t)\in\mathbb{R}^N\times(0,\infty),
$$
and we deduce by passing to the limit in the weak formulation of Eq. \eqref{eq1} and employing the monotone convergence theorem that $U_K$ is a solution to Eq. \eqref{eq1} with initial trace $K|x|^{(\sigma+2)/(m-p)}$. We omit the (standard) details, that can be found in and easily adapted from \cite[Theorem 2.1 and Section 3]{KPV85}. Moreover, the same comparison principle proves that, if $U$ is any other solution to the Cauchy problem
\begin{equation}\label{cpK}
	\left\{\begin{array}{ll}u_t=\Delta u^m-|x|^{\sigma}u^p, & (x,t)\in \mathbb{R}^N\times(0,\infty),\\
		u(x,0)=K|x|^{(\sigma+2)/(m-p)}, & x\in \mathbb{R^N},
	\end{array}\right.
\end{equation}
then $U(x,t)\geq U_{n,K}(x,t)$ for any $n\geq1$ and thus $U(x,t)\geq U_K(x,t)$ for any $(x,t)\in \mathbb{R}^N\times(0,\infty)$, which proves that $U_K$ is a minimal solution to the Cauchy problem \eqref{cpK}. 

We are left to prove that $U_K$ is self-similar. But this follows straightforwardly by a scaling argument (also employed and given in detail in \cite[Lemma 3.3]{KPV85}). We introduce the scaling transformation
\begin{equation}\label{scaling}
u_{\lambda}(x,t)=(\mathcal{T}_{\lambda}u)(x,t):=\lambda^{\alpha}u(\lambda^{-\beta}x,\lambda t), \quad \lambda\in(0,\infty).
\end{equation}
Since 
\begin{equation*}
	\begin{split}
&\partial_tu_{\lambda}(x,t)=\lambda^{\alpha+1}\partial_tu(\lambda^{-\beta}x,\lambda t),\\	
&\Delta u_{\lambda}^m(x,t)=\lambda^{m\alpha-2\beta}\Delta u^m(\lambda^{-\beta}x,\lambda t),\\
&|x|^{\sigma}u_{\lambda}^p=\lambda^{\alpha p+\beta\sigma}|\lambda^{-\beta}x|^{\sigma}u^p(\lambda^{-\beta}x,\lambda t)	
	\end{split}
\end{equation*}
and 
$$
\alpha+1=m\alpha-2\beta=p\alpha+\beta\sigma=\frac{m\sigma+2p}{\sigma(m-1)+2(p-1)},
$$
we deduce that $\mathcal{T}_{\lambda}u$ is a solution to Eq. \eqref{eq1}, for any solution $u$ to Eq. \eqref{eq1} and any $\lambda\in(0,\infty)$. Moreover, the initial condition $K|x|^{(\sigma+2)/(m-p)}$ is also invariant to the scaling \eqref{scaling}. It then follows that the scaling \eqref{scaling} must transform minimal solutions into minimal solutions, and thus the solution $U_K$ constructed by approximation via the problems \eqref{cpKn} is invariant to $\mathcal{T}_{\lambda}$, which is equivalent to be in self-similar form, completing the proof for $K\in(0,K_0)$. 

\medskip

\noindent \textbf{Step 3. End of the proof.} The argument in Step 1 proves that this existence can be continued above $K_0$, in some right neighborhood $(K_0,K_0+\epsilon)$ (for some $\epsilon$ depending on the parameter $\delta$ in Step 1). This proves that $K^*>K_0$, as claimed. Finally, once a trajectory representing in the alternative formulation \eqref{PSchange}-\eqref{PSsyst} a profile satisfying \eqref{beh.zero} as $\xi\to0$ and \eqref{beh.inf} as $\xi\to\infty$ crosses the plane $Y=0$ towards the negative half-space $Y<0$, it cannot re-enter the half-space $Y>0$ afterwards. Indeed, crossing the plane $Y=0$ from its negative side to its positive side implies directly entering the positively invariant region $\mathcal{R}$, as it follows from Proposition \ref{prop.inv}. Thus, all these trajectories contain a single point with $Y=0$ and it follows that the corresponding profiles have a single maximum point, completing the proof.
\end{proof}	

\bigskip

\noindent \textbf{Acknowledgements} This work is partially supported by the Spanish project PID2024-160967NB-I00 funded by Agencia Estatal de Investigaci\'on (Spain).

\bigskip

\noindent \textbf{Data availability} Our manuscript has no associated data.

\bigskip

\noindent \textbf{Conflict of interest} The authors declare that there is no conflict of interest.

\bibliographystyle{plain}

\end{document}